\documentclass{siamltex}%
\usepackage{amssymb}
\usepackage{amsmath}
\usepackage{sw20siam}
\usepackage{rotating}
\usepackage{graphicx}
\usepackage{caption}%
\usepackage{amsfonts}
\providecommand{\U}[1]{\protect\rule{.1in}{.1in}}

\newtheorem{example}[theorem]{Example}

\begin{document}

\title{Characterizing weak solutions for vector optimization problems }
\author{N. Dinh\thanks{ International University, Vietnam National University - HCMC,
Linh Trung ward, Thu Duc district, Ho Chi Minh city (ndinh@hcmiu.edu.vn). }
\and M.A. Goberna\thanks{Department of Mathematics, University of Alicante,
Alicante, Spain (mgoberna@ua.es).}
\and D.H. Long\thanks{ Tien Giang University, 119 Ap Bac street, My Tho city, Tien
Giang province, Vietnam\newline(danghailong88@gmail.com).}
\and M.A. L\'{o}pez\thanks{Department of Mathematics, University of Alicante,
Alicante, Spain, and CIAO, Federation University, Ballarat, Australia
(marco.antonio@ua.es).} }
\maketitle
\date{4 February 2016}

\begin{abstract}
This paper provides characterizations of the weak solutions of optimization
problems where a given vector function $F,$ from a decision space $X$ to an
objective space $Y$, is "minimized" on the set of elements $x\in C$ (where
$C\subset X$ is a given nonempty constraint set), satisfying $G\left(
x\right)  \leqq_{S}0_{Z},$ where $G$ is another given vector function from $X
$ to a constraint space $Z$ with positive cone $S$. The three spaces $X,Y,$
and $Z$ are locally convex Hausdorff topological vector spaces, with $Y$ and
$Z$ partially ordered by two convex cones $K$ and $S,$ respectively, and
enlarged with a greatest and a smallest element. In order to get suitable
versions of the Farkas lemma allowing to obtain optimality conditions
expressed in terms of the data, the triplet $\left(  F,G,C\right)  ,$ we use
non-asymptotic representations of the $K-$epigraph of the conjugate function
of $F+I_{A},$ where $I_{A}$ denotes the indicator function of the feasible set
$A,$ that is, the function associating the zero vector of $Y$ to any element
of $A$ and the greatest element of $Y$ to any element of $X\diagdown A.$

\end{abstract}
\keyphrases{Vector optimization, weak minimal solutions, qualification conditions,
Farkas-type results for vector functions, strong duality}
\AMclass{58E17, 90C29, 90C46, 49N15}

\section{Introduction}

This paper analyzes vector optimization problems of the form%
\[%
\begin{array}
[c]{rrl}%
\mathrm{(VOP)}\ \  & \operatorname*{WMin}\left\{  F(x):x\in C,\;G(x)\leqq
_{S}0_{Z}\right\}  , &
\end{array}
\]
where $\operatorname*{WMin}$ stands for the task consisting of determining the
weakly minimal elements (concept defined in Section 2) of some subset of the
objective space $Y,$ which is equipped with a partial ordering $\leqq_{K}$
induced by a convex cone $K,$ the constraint set $C$ is a given subset of the
decision space $X,$ $\leqq_{S}$ denotes the partial ordering induced in the
constraint space $Z$ by a convex cone $S,$ $0_{Z}$ is the zero vector in $Z,$
$F\colon X\rightarrow{Y}\cup\left\{  +\infty_{Y}\right\}  $ and $G\colon
X\rightarrow Z\cup\left\{  +\infty_{Z}\right\}  ,$ with $+\infty_{Y} $ and
$+\infty_{Z}$ denoting "greatest elements" aggregated to $Y$ and $Z,$
respectively. We assume that $X,Y,Z$ are locally convex Hausdorff topological
vector spaces. The main purpose of this paper is to give conditions for a
given $\overline{x}\in A$ to satisfy $F\left(  \overline{x}\right)  \in$
$\operatorname*{WMin}F(A),$ where $A:=\{x\in C,\ G(x)\leqq_{S}0_{Z}\}$ is the
feasible set of $\mathrm{(VOP).}$ Moreover, we are particularly interested in
conditions which are expressed in terms of the data, that is, in conditions
only involving mathematical objects related to the triplet $\left(
F,G,C\right)  .$

Different concepts of solutions in vector optimization have been proposed,
each one having its own set of advantages and disadvantages. For instance,
regarding multiobjective optimization (when $Y=\mathbb{R}^{m}$ and
$K=\mathbb{R}_{+}^{m}$), it is usually admitted that weakly efficient
solutions, efficient solutions, and super efficient solutions are preferable
from the computational, practical, and stability perspectives, respectively
(see, e.g., \cite{BM09}, \cite{BZ93}, \cite{BGW09}, \cite{HL11}, \cite{JJ04},
\cite{Loh11}, and references therein). On the other hand, weak orders allow to
apply the elegant conjugate duality machinery (\cite{Bot09}). So,
computability and mathematical elegance are the main reasons to consider in
this paper weak orders and weak minimal solutions.

Consider now, for comparison purposes, a scalar optimization problem of the
form%
\[%
\begin{array}
[c]{rrl}%
\mathrm{(SOP)}\ \  & \operatorname*{Min}\left\{  f(x):x\in C,\;G(x)\leqq
_{S}0_{Z}\right\}  , &
\end{array}
\]
where $\operatorname*{Min}$ is the task of finding the minimum, if it exists,
of some subset of $Y:=\mathbb{R},$ with positive cone $\mathbb{R}_{+},$ whose
objective function is $f\colon X\rightarrow\mathbb{R\cup}\left\{  \pm
\infty\right\}  .$ More precisely, the optimality conditions for
$\mathrm{(SOP)}$ allow to determine those $\overline{x}\in A$ such that
$f\left(  \overline{x}\right)  =\min f\left(  A\right)  .$ A classical way of
handling $\mathrm{(SOP)}$ consists of reformulating this problem as an
unconstrained scalar one with objective function $f+i_{A},$ where $i_{A}$
denotes the indicator function of the feasible set $A,$ i.e., $i_{A}\left(
x\right)  =0$ if $x\in A$ and $i_{A}\left(  x\right)  =+\infty$ if $x\in
X\diagdown A.$ Optimality conditions involving $A$ (as the classical one that
the null functional is a subgradient of $f+i_{A}$ at $\overline{x}$) are
considered too abstract for practical purposes as a manageable description of
$A$ is seldom available. The epigraph of the Fenchel-Moreau conjugate of
$f+i_{A},$ denoted by $\operatorname*{epi}\left(  f+i_{A}\right)  ^{\ast},$
plays an important role in order to obtain checkable optimality conditions for
$\mathrm{(SOP),}$ that is, conditions which are expressed in terms of the
data, the triple $\left(  f,G,C\right)  $. This is usually done through the
introduction of qualification conditions on $\left(  f,G,C\right)  $ (called
constraint qualifications when they only involve $G$ and $C$)\ allowing to
represent $\operatorname*{epi}\left(  f+i_{A}\right)  ^{\ast}$ in terms of
$\left(  f,G,C\right)  $ (see, e.g., \cite[Theorem 8.2]{Bot09}).

In the same vein, in this paper we reformulate $\mathrm{(VOP)}\ $as an
unconstrained vector optimization problem with objective function $F+I_{A},$
where $I_{A}$\ denotes the indicator function of the feasible set $A,$ i.e.,
$I_{A}\left(  x\right)  =0_{Y}$ if $x\in A$ and $i_{A}\left(  x\right)
=+\infty_{Y}$ if $x\in X\diagdown A.$ After the introductory Section 2,
Section 3 provides different representations of the $K-$epigraph of the
conjugate of $F+I_{A},$ say $\operatorname*{epi}{}_{K}\left(  F+I_{A}\right)
^{\ast}$ (concept to be introduced in Section 2) in terms of $\left(
F,G,C\right)  $. These representations are called asymptotic when they involve
a limiting process (typically, in the form of closure of some set depending on
$F,$ $G,$ and $C$) and non-asymptotic otherwise. The main results in this
paper are Theorems \ref{theo2} and \ref{theovectorw}, which are alternative
extensions of \cite[Theorem 8.2]{Bot09} to the vector framework providing
non-asymptotic representations of $\operatorname*{epi}\nolimits_{K}\left(
F+I_{A}\right)  ^{\ast}$ under the same qualification conditions (Theorems
\ref{theovector2} and \ref{theovectorw2}). We provide in Section 4 two kinds
of Farkas-type results oriented to $\mathrm{(VOP)}$: those which characterize
the inclusion of $A$ in a second subset $B$ of $X$ depending on $F$\ under
certain set of assumptions $P$ are said to be \emph{Farkas lemmas} while those
establishing the equivalence of $P$ with some characterization of the
inclusion $A\subset B$ are called \emph{characterizations of Farkas lemma}.
Similarly, the final Section 5 provides \emph{optimality conditions}
establishing characterizations of the weakly minimal solutions to
$\mathrm{(VOP)}$ under $P$ and \emph{characterizations of optimality
conditions} asserting the equivalence of $P$ with some optimality conditions.
A strong duality theorem for $\mathrm{(VOP)}$ is obtained from the optimality
conditions. The results in Sections 4 and 5 can be seen as non-abstract
versions of the corresponding results in \cite[Sections 4 and 5]{DGLM16},
where $P$ involves the feasible set $A.$

\section{Preliminaries}

Throughout the paper $X,Y,Z$ are three given locally convex Hausdorff
topological vector spaces with topological dual spaces denoted by $X^{\ast
},Y^{\ast},Z^{\ast}$, respectively. The only topology we consider on dual
spaces is the weak*-topology.

For a set $U\subset X$, we denote by $\operatorname*{cl}U$,
$\operatorname*{conv}U$, $\operatorname*{cl}\operatorname*{conv}U$,
$\operatorname*{lin}U$, $\operatorname*{int}U,$ $\operatorname*{ri}U,$ and
$\operatorname*{sqri}U$ the \emph{closure}, the \emph{convex hull}, the
\emph{closed convex hull}, the \emph{linear hull}, the \emph{interior}, the
\emph{relative interior}, and the \emph{strong quasi-relative interior} of
$U,$ respectively. Note that $\operatorname*{cl}\operatorname*{conv}%
U=\operatorname*{cl}(\operatorname*{conv}U)$. The null vector in $X$ is
denoted by $0_{X}$ and the dimension of a linear subspace $U$ of $X$ by $\dim
U.$ Given two subsets $A$ and $B$ of a topological space, one says that $A$ is
\emph{closed regarding} $B$ if $B\cap\operatorname*{cl}A=B\cap A$ (see, e.g.
\cite[Section 9]{Bot09})$.$

We assume that $K$ is a given closed, pointed, convex cone in $Y$ with
nonempty interior, i.e., $\operatorname*{int}K\neq\emptyset$. A \emph{weak
ordering} in $Y$, "$<_{K}$", is defined as follows: for $y_{1},y_{2}\in Y$,
\[
y_{1}<_{K}y_{2}\quad\Longleftrightarrow\quad y_{1}-y_{2}\in
-\operatorname*{int}K,
\]
or equivalently, $y_{1}\not < _{K}y_{2}$ if and only if $y_{1}-y_{2}%
\notin-\operatorname*{int}K$.

We enlarge $Y$ by attaching a \emph{greatest element} $+\infty_{Y}$ and a
\emph{smallest element} $-\infty_{Y}$ with respect to $<_{K}$, which do not
belong to $Y$, and we denote $Y^{\bullet}:=Y\cup\{-\infty_{Y},+\infty_{Y}\}$.
By convention, $-\infty_{Y}<_{K}y<_{K}+\infty_{Y}$ for any $y\in Y $. We also
assume by convention that
\begin{gather}
-(+\infty_{Y})=-\infty_{Y},\qquad\qquad-(-\infty_{Y})=+\infty_{Y},\nonumber\\
(+\infty_{Y})+y=y+(+\infty_{Y})=+\infty_{Y},\quad\forall y\in Y\cup
\{+\infty_{Y}\},\label{conv}\\
(-\infty_{Y})+y=y+(-\infty_{Y})=-\infty_{Y},\quad\forall y\in Y\cup
\{-\infty_{Y}\}.\nonumber
\end{gather}
The sums $(-\infty_{Y})+(+\infty_{Y})$ and $(+\infty_{Y})+(-\infty_{Y})$ are
not considered in this paper.

Notice that in the space $Y$ the cone $K$ also generates another order
$\leqq_{K}$ defined as, for $y_{1},y_{2}\in Y$,
\[
y_{1}\leqq_{K}y_{2}\ \text{ if and only if }\ y_{2}\in y_{1}+K.
\]
It is obvious that the order $\leqq_{K}$ also can be extended to $Y^{\bullet}$
with the convention that $-\infty_{Y}\leqq_{K}y\leqq_{K}+\infty_{Y}$ for any
$y\in Y$ together with the others in \eqref{conv}.

We now recall the following basic definitions regarding the subsets of
$Y^{\bullet}$ (see, e.g., \cite{Bot09}, \cite[Definition 7.4.1]{BGW09},
\cite{Gr15}, \cite{JJ04}, \cite{KTZ05}, \cite{Tan92}, etc.):\medskip

\begin{definition}
\label{def1}Consider a set $M$ such that $\emptyset\neq M\subset Y^{\bullet}%
.$\newline1. An element $\bar{v}\in Y^{\bullet}$ is said to be a \emph{weakly
infimal element} of $M$ if for all $v\in M$ we have $v\not < _{K}\bar{v}$ and
if for any $\tilde{v}\in Y^{\bullet}$ such that $\bar{v}<_{K}\tilde{v}$ there
exists some $v\in M$ satisfying $v<_{K}\tilde{v}$. The set of all weakly
infimal elements of $M$ is denoted by $\operatorname*{WInf}M$ and is called
the \emph{weak infimum} of $M$.\newline2. An element $\bar{v}\in Y^{\bullet}$
is said to be a \emph{weakly supremal element} of $M$ if for all $v\in M$ we
have $\bar{v}\not < _{K}v$ and if for any $\tilde{v}\in Y^{\bullet}$ such that
$\tilde{v}<_{K}\bar{v}$ there exists some $v\in M$ satisfying $\tilde{v}%
<_{K}v$. The set of all weakly supremal elements of $M$ is denoted by
$\operatorname*{WSup}M$ and is called the \emph{weak supremum} of $M$%
.\newline3. The \emph{weak minimum} of $M$ is the set $\operatorname*{WMin}%
M=M\cap\operatorname*{WInf}M$ and its elements are the \emph{weakly minimal
elements} of $M$.\newline4. The \emph{weak maximum} of $M$ is the set
$\operatorname*{WMax}M=M\cap\operatorname*{WSup}M$ and its elements are the
\emph{weakly maximal elements} of $M$.\newline5. An element $\overline{v}\in
M$ is called a \emph{strongly maximal element} of $M$ if it holds $v\leqq
_{K}\overline{v}$ for all $v\in M$. The set of all strongly maximal elements
of $M$ is denoted by $\operatorname*{SMax}M$.\medskip
\end{definition}

Observe that, if $M\subset Y,$ then $\overline{v}\in\operatorname*{SMax}M$ if
and only if $M\subset\overline{v}-K$. Thus, if $M\subset Y$ then
\begin{equation}
\operatorname*{SMax}M=\{\bar{v}\in M:M\subset\bar{v}-K\}.\label{2.4}%
\end{equation}
Moreover, in this case, if $K$ is a pointed cone and $\operatorname*{SMax}%
M\neq\emptyset$ then $\operatorname*{SMax}M$ is a singleton, i.e., the
strongly maximum element of the set $M$ in this case, if exists, will be
unique. In such a case, we write $\overline{v}=\operatorname*{SMax}M$ instead
of $\operatorname*{SMax}M=\{\overline{v}\}$.

The next elementary properties will be used in the sequel:

$\bullet$ According to the definition above,
\begin{equation}%
\begin{array}
[c]{ll}%
+\infty_{Y}\in\operatorname*{WSup}M\; & \Longleftrightarrow
\;\operatorname*{WSup}M=\{+\infty_{Y}\}\;\\
& \Longleftrightarrow\;\forall\tilde{v}\in Y,\;\exists v\in M:\tilde{v}<_{K}v.
\end{array}
\label{wsup=infty}%
\end{equation}

$\bullet$ If $\emptyset\neq M\subset Y$ and $\operatorname*{WSup}%
M\neq\{+\infty_{Y}\}$, by \cite[Proposition 2.7(i)]{DGLM16}, one has
\begin{equation}
\operatorname*{WSup}M=\operatorname*{cl}(M-\operatorname*{int}K)\setminus
(M-\operatorname*{int}K).\label{formularwsup}%
\end{equation}

$\bullet$ If $M\neq\emptyset$ and $+\infty_{Y}\not \in M\neq\{-\infty_{Y}\}$
then, by \cite[Proposition 2.7(ii)]{DGLM16}, one gets
\begin{equation}
\operatorname*{WMax}M=M\setminus(M-\operatorname*{int}K).\label{eqmax}%
\end{equation}

$\bullet$ If $\emptyset\neq M\subset Y$ and $\operatorname*{WSup}%
M\neq\{+\infty_{Y}\}$, from \cite[Proposition 2.4]{Tan92}, one gets
\begin{equation}
\operatorname*{WSup}M-\operatorname*{int}K=M-\operatorname*{int}K.\label{1.3}%
\end{equation}

We shall also use the next lemma.\medskip

\begin{lemma}
\label{lem1.1} Given $\emptyset\neq M\subset Y^{\bullet},$ the following
statements hold true:\newline$(i)$\textrm{\ }If $+\infty_{Y}\not \in M$ and
$M\cap\operatorname*{int}K=\emptyset,$ then $\operatorname*{WSup}%
M\neq\{+\infty_{Y}\};$\textrm{\newline}$(ii)$\textrm{\ }If there exists
$v_{0}\in M\cap\operatorname*{int}K$ such that $\lambda v_{0}\in M$ for all
$\lambda>0,$ then $\operatorname*{WSup}M=\{+\infty_{Y}\};$\textrm{\ \newline%
}$(iii)$\textrm{\ }If $M\subset-K$ and $0_{Y}\in M$ then $\operatorname*{WSup}%
M=\operatorname*{WSup}(-K)$.\textrm{\ }
\end{lemma}

\begin{proof}
$(i)$ Assume that $M\cap\operatorname*{int}K=\emptyset$. Then, $0_{Y}%
\not < _{K}v$ for all $v\in M$ and it follows from \eqref{wsup=infty} that
$\operatorname*{WSup}M\neq\{+\infty_{Y}\}$.

$(ii)$ Assume that there is $v_{0}\in M\cap\operatorname*{int}K$ such that
$\lambda v_{0}\in M$ for all $\lambda>0$. If $\operatorname*{WSup}%
M\neq\{+\infty_{Y}\}$, then by \eqref{wsup=infty}, there exists $\tilde{v}\in
Y$ such that $\tilde{v}\not < _{K}v$ for any $v\in M$. We get
\begin{align}
\tilde{v}\not < _{K}\lambda v_{0},\;\forall\lambda>0\Longrightarrow\; &
\lambda v_{0}-\tilde{v}\not \in \operatorname*{int}K,\;\forall\lambda
>0\nonumber\\
\Longrightarrow\; &  v_{0}-\frac{1}{\lambda}\tilde{v}\not \in
\operatorname*{int}K,\;\forall\lambda>0.\label{eq2}%
\end{align}
On the other hand, because of the continuity of the map $t\mapsto
v_{0}-t\tilde{v}$ at $0$ and $v_{0}\in\operatorname*{int}K$, there exists
$\epsilon>0$ such that
\[
t\in\left]  -\epsilon,\epsilon\right[  \;\Longrightarrow\;v_{0}-t\tilde{v}%
\in\operatorname*{int}K,
\]
which contradicts \eqref{eq2} and the proof is complete.

$(iii)$ Assume that $M\subset-K$ and $0_{Y}\in M$. Then $M-\operatorname*{int}%
K=-\operatorname*{int}K$. Indeed, $M-\operatorname*{int}K\subset
-K-\operatorname*{int}K=-\operatorname*{int}K$. Since $0_{Y}\in M$, we also
have $-\operatorname*{int}K\subset M-\operatorname*{int}K$. On other hand,
because $K$ is a pointed cone, $M\subset-K$ yields $M\cap\operatorname*{int}%
K=\emptyset$. So we get from $(i)$ that $\operatorname*{WSup}M\neq\{+\infty
\}$. According to \eqref{formularwsup},
\[
\operatorname*{WSup}M=\operatorname*{cl}(M-\operatorname*{int}K)\setminus
(M-\operatorname*{int}K)=\operatorname*{cl}(-\operatorname*{int}%
K)\setminus(-\operatorname*{int}K)=\operatorname*{WSup}(-K)
\]
and we are done. \medskip
\end{proof}

\begin{lemma}
\label{pro2.10} Assume $\emptyset\neq M\subset Y,$ $\operatorname*{WSup}%
M\subset Y,$ and there exists $v_{0}\in Y\setminus(-K)$ such that $\lambda
v_{0}\in M $ for all $\lambda>0$. Then $\operatorname*{SMax}%
(\operatorname*{WSup}M)=\emptyset$.
\end{lemma}

\begin{proof}
Let us suppose by contradiction that $\operatorname*{SMax}%
(\operatorname*{WSup}M)\neq\emptyset$ and $\bar{v}=\operatorname*{SMax}%
(\operatorname*{WSup}M)$. Since $\operatorname*{WSup}M\subset Y,$ one has
$\operatorname*{WSup}M\subset\bar{v}-K$ (see (\ref{2.4})). It follows from
\eqref{formularwsup} and \eqref{1.3} that
\begin{align*}
M\subset\operatorname*{cl}(M-\operatorname*{int}K) &  =[\operatorname*{cl}%
(M-\operatorname*{int}K)\setminus(M-\operatorname*{int}K)]\cup
(M-\operatorname*{int}K)\\
&  =\operatorname*{WSup}M\cup(\operatorname*{WSup}M-\operatorname*{int}K)\\
&  \subset\operatorname*{WSup}M-K,
\end{align*}
and consequently, $M\subset\bar{v}-K-K=\bar{v}-K.$ Thus, from the assumption
that $\lambda v_{0}\in M$ for all $\lambda>0$, one has
\begin{equation}
\lambda v_{0}\in\bar{v}-K,\;\forall\lambda>0\Longrightarrow\;v_{0}-\frac
{1}{\lambda}\bar{v}\in-K,\;\forall\lambda>0.\label{2.14}%
\end{equation}
On other hand, because $v_{0}\in Y\setminus(-K)$, the set $-K$ is closed, and
the map $t\mapsto v_{0}-t\bar{v}$ is continuous at $t=0$, we can find
$\epsilon>0$ such that
\[
t\in]-\epsilon,\epsilon\lbrack\;\Longrightarrow\;v_{0}-t\bar{v}\in
Y\setminus(-K),
\]
which contradicts \eqref{2.14} and the proof is complete. \medskip
\end{proof}

We denote by $\mathcal{L}(X,Y)$ the space of linear continuous mappings from
$X$ to $Y$, and by $0\mathbf{{_{\mathcal{L}}}}\in\mathcal{L}(X,Y)$ the zero
mapping defined by $0\mathbf{{_{\mathcal{L}}}}(x)=0_{Y}$ for all $x\in X$.
Obviously, $\mathcal{L}(X,Y)=X^{\ast\text{ }}$whenever $Y=\mathbb{R}.$ We
consider $\mathcal{L}(X,Y)$ equipped with the so-called \emph{weak topology},
that is, the one defined by the pointwise convergence. In other words, given a
net $(L_{i})_{i\in I}\subset\mathcal{L}(X,Y)$ and $L\in\mathcal{L}(X,Y)$,
$L_{i}\rightarrow L$ means that $L_{i}(x)\rightarrow L(x)$ in $Y$ for all
$x\in X$.

Given a vector-valued mapping $F\colon X\rightarrow Y^{\bullet}$, the
\emph{domain} of $F$ is defined by
\[
\operatorname{dom}F:=\{x\in X:F\left(  x\right)  \neq+\infty_{Y}\}
\]
and $F$ is \emph{proper} when $\operatorname{dom}F\neq\emptyset$ and
$-\infty_{Y}\notin F(X)$. The \emph{$K$-epigraph} of $F$, denoted by
$\operatorname*{epi}{}_{K}F$, is defined by
\[
\operatorname*{epi}{}_{K}F:=\{(x,y)\in X\times Y:y\in F(x)+K\}.
\]
We say that $F$ is $K-$\emph{convex} ($K-$\emph{epi closed}\textit{)} if
$\operatorname*{epi}{}_{K}\,F$ is a convex set (a closed set in the product
space, respectively). If $F$ is $K-$convex, it is evident that
$\operatorname{dom}F$ is a convex set in $X.$\medskip

\begin{definition}
\label{def1.2} The set-valued map $F^{\ast}\colon\mathcal{L}%
(X,Y)\rightrightarrows Y^{\bullet}$ defined by
\[
F^{\ast}(L):=\operatorname*{WSup}\{L(x)-F(x):x\in X\}=\operatorname*{WSup}%
\{(L-F)(X)\},
\]
is called the \emph{conjugate map} of $F$. The \emph{domain} and the
(\emph{strong}) \emph{"max-domain"} of $F^{\ast}$ are defined as%
\[
\operatorname{dom}F^{\ast}:=\big\{L\in\mathcal{L}(X,Y)\,:\,F^{\ast}%
(L)\neq\{+\infty_{Y}\}\big\},
\]
and
\[
\operatorname{dom}_{M}F^{\ast}:=\big\{L\in\mathcal{L}(X,Y)\,:\,F^{\ast
}(L)\subset Y\text{ and }\operatorname*{SMax}F^{\ast}(L)\neq\emptyset\},
\]
respectively, while the \emph{$K$-epigraph} of $F^{\ast}$ is
\[
\operatorname*{epi}{}_{K}F^{\ast}:=\big\{(L,y)\in\mathcal{L}(X,Y)\times Y:y\in
F^{\ast}(L)+K\big\}.
\]

\end{definition}

Let $S$ be a nonempty closed and convex cone in $Z$ and $\leqq_{S}$ be the
ordering on $Z$\ induced by the cone $S$, i.e.,
\begin{equation}
z_{1}\leqq_{S}z_{2}\ \text{ if and only if }\ z_{2}-z_{1}\in S.\label{order}%
\end{equation}
We also enlarge $Z$ by attaching a greatest element $+\infty_{Z}$ and a
smallest element $-\infty_{Z}$ (with respect to $\leqq_{S}$) which do not
belong to $Z$, and define $Z^{\bullet}:=Z\cup\{-\infty_{Z},\ +\infty_{Z}\}$.
In $Z^{\bullet}$ we adopt the same conventions as in \eqref{conv}.

For $T\in\mathcal{L}(Z,Y)$ and $G\colon X\rightarrow Z\cup\{+\infty_{Z}\}$, we
define the composite function $T\circ G\colon X\rightarrow{Y}^{\bullet}$ as
follows:
\[
(T\circ G)(x):=\left\{
\begin{array}
[c]{ll}%
T(G(x)), & \text{if }G(x)\in Z,\\
+\infty_{Y}, & \text{if $G(x)=+\infty_{Z}.$}%
\end{array}
\right.
\]

Recall that $S$ is a nonempty closed and convex cone in $Z$. Let us set%
\[
\mathcal{L}_{+}(S,K):=\{T\in\mathcal{L}(Z,Y):\ T(S)\subset K\}\ \text{ }%
\]
and
\[
\mathcal{L}_{+}^{w}(S,K):=\{T\in\mathcal{L}(Z,Y):T(S)\cap(-\operatorname*{int}%
K)=\emptyset\}.
\]

It is clear that $\mathcal{L}_{+}(S,K)\subset\mathcal{L}_{+}^{w}(S,K)$.
Indeed, for any $T\in\mathcal{L}_{+}(S,K)$, one has $T(S)\subset K$. So,
$T(S)\cap(-\operatorname*{int}K)=\emptyset$ (as $K$ is pointed cone) and
hence, $T\in\mathcal{L}_{+}^{w}(S,K)$. However, the inclusion $\mathcal{L}%
_{+}(S,K)\subset\mathcal{L}_{+}^{w}(S,K)$ is generally strict (see Example
\ref{Example1} below).

It is worth noticing that, when $Y=\mathbb{R}$ and $K=\mathbb{R}_{+},$ the
conjugate, the domain and the \emph{$K$-}epigraph of $f:X\longrightarrow
\mathbb{R\cup}\left\{  +\infty\right\}  $ are nothing else than the ordinary
conjugate, the domain, and the epigraph of the scalar function $f$, i.e.,
\[
f^{\ast}\left(  x^{\ast}\right)  :=\sup_{x\in X}\left(  \left\langle x^{\ast
},x\right\rangle -f\left(  x\right)  \right)  ,\ \forall x^{\ast}\in X^{\ast},
\]%
\[
\operatorname*{dom}f:=\{x\in X:\ f(x)\neq+\infty\},
\]
and
\[
\operatorname*{epi}f:=\{(x,r)\in X\times\mathbb{R}:\ x\in\operatorname*{dom}%
f,\ f(x)\leq r\},
\]
respectively. Moreover, since
\[
T(S)\cap(-\operatorname*{int}\mathbb{R}_{+})=\emptyset\quad\Longleftrightarrow
\quad T(S)\subset\mathbb{R}_{+},
\]
we have
\[
\mathcal{L}_{+}^{w}(S,\mathbb{R}_{+})=\mathcal{L}_{+}(S,\mathbb{R}_{+}%
)=S^{+}:=\{z^{\ast}\in Z^{\ast}:\ \langle z^{\ast},s\rangle\geq0\ \text{for
all}\ s\in S\},
\]
in other words, the (positive) \textit{dual cone} $S^{+}$ of $S$ in the sense
of convex analysis.

In order to obtain a suitable interpretation of $\mathcal{L}_{+}^{w}(S,K)$ we
must extend the concept of indicator function from scalar to vector functions:
the \emph{indicator map} $I_{D}:X\rightarrow{Y}^{\bullet}$ of a set $D\subset
X$ is defined by
\[
I_{D}(x)=\left\{
\begin{array}
[c]{ll}%
0_{Y}, & \text{if }x\in D,\\
+\infty_{Y}, & \text{otherwise.}%
\end{array}
\right.
\]
In the case $Y=\mathbb{R}$, $I_{D}$ is the usual indicator function $i_{D}%
$.\medskip

\begin{lemma}
\label{lem4.1} One has
\[
\mathcal{L}_{+}^{w}(S,K)=\operatorname{dom}I_{-S}^{\ast}\text{ and
}\,\mathcal{L}_{+}(S,K)=\operatorname{dom}_{M}I_{-S}^{\ast}.
\]

\end{lemma}

\begin{proof}
$\bullet$ Taking an arbitrary $T\in\mathcal{L}(Z,Y)$, one has
\begin{equation}
I_{-S}^{\ast}(T)=\operatorname*{WSup}\{T(z):z\in-S\}=\operatorname*{WSup}%
T(-S),\label{2.3}%
\end{equation}
and so, $T\in\operatorname{dom}I_{-S}^{\ast}$ if and only if
$\operatorname*{WSup}T(-S)\neq\{+\infty_{Y}\}$. Two cases are possible.

If $T\in\mathcal{L}_{+}^{w}(S,K)$, $T(S)\cap(-\operatorname*{int}K)=\emptyset
$, and consequently, $T(-S)\cap\operatorname*{int}K=\emptyset$. So, it follows
from Lemma \ref{lem1.1}$(i)$ that $\operatorname*{WSup}T(-S)\neq\{+\infty
_{Y}\}$.

If $T\in\mathcal{L}(X,Y)\setminus\mathcal{L}_{+}^{w}(S,K)$, there exists
$v_{0}\in T(S)\cap(-\operatorname*{int}K)$. Then, $-v_{0}\in
\operatorname*{int}K$ and $\lambda(-v_{0})\in T(-S)$ for all $\lambda>0$
because $S$ is a cone. So, by Lemma \ref{lem1.1}$(ii)$, $\operatorname*{WSup}%
T(-S)=\{+\infty_{Y}\}$. Consequently, $\operatorname{dom}I_{-S}^{\ast
}=\mathcal{L}_{+}^{w}(S,K)$.

\medskip$\bullet$ Take an arbitrary $T\in\mathcal{L}_{+}(S,K)$. Then one has
$T(S)\subset K$, or equivalently, $T(-S)\subset-K$. It is clear that
$0_{Y}=T(0_{Z})\in T(-S)$. According to Lemma \ref{lem1.1}$(iii)$,
$I_{-S}^{\ast}(T)=\operatorname*{WSup}T(-S)=\operatorname*{WSup}(-K)$. So,
$\operatorname*{SMax}I_{-S}^{\ast}(T)=\{0_{Y}\}\neq\emptyset$, and
consequently, $T\in\operatorname{dom}_{M}I_{-S}^{\ast}$. \newline\indent Now
take an arbitrary $T\in\mathcal{L}(Z,Y)\diagdown\mathcal{L}_{+}(S,K)$. One has
$T(S)\not \subset K$, or equivalently, there exists $s_{0}\in-S$ such that
$T(s_{0})\notin-K.$ Thus, applying Lemma \ref{pro2.10} with $M=T(-S)$ and
$v_{0}=T(s_{0})$, we get that, if $\operatorname*{WSup}T(-S)\subset Y,$ then
\[
\operatorname*{SMax}\left[  \operatorname*{WSup}T(-S)\right]
=\operatorname*{SMax}\left[  I_{-S}^{\ast}\left(  T\right)  \right]
=\emptyset.
\]
So, $T\notin\operatorname{dom}_{M}I_{-S}^{\ast}$ and we are done. \medskip
\end{proof}

We shall use the following simple example for illustrative purposes throughout
the paper.\medskip

\begin{example}
\label{Example1} Take $X=Z=\mathbb{R}$, $Y=\mathbb{R}^{2}$, $K=\mathbb{R}%
_{+}^{2}$, $S=\mathbb{R}_{+},$ $F:\mathbb{R}\longrightarrow\mathbb{R}^{2}%
$\ the null mapping,\ and $G:\mathbb{R}\longrightarrow\mathbb{R}$\ such that
$G\left(  x\right)  =-x$ for all $x\in\mathbb{R}.$ Then $\mathcal{L}%
(Z,Y)=\mathbb{R}^{2},$ $\mathcal{L}_{+}(S,K)=\mathbb{R}_{+}^{2},$ and
$\mathcal{L}_{+}^{w}(S,K)=\left\{  \left(  t_{1},t_{2}\right)  \in
\mathbb{R}^{2}:t_{1}\geq0\vee t_{2}\geq0\right\}  .$ Moreover, given $\left(
\alpha,\beta\right)  \in\mathbb{R}^{2},$%
\[
F^{\ast}\left(  \alpha,\beta\right)  =\operatorname*{WSup}\{\mathbb{R}\left(
\alpha,\beta\right)  \}=\left\{
\begin{array}
[c]{ll}%
\left[  (-\mathbb{R}_{+})\times\{0\}\right]  \cup\left[  \{0\}\times
(-\mathbb{R}_{+})\right]  , & \text{if }\alpha=\beta=0,\\
\left\{  +\infty_{\mathbb{R}^{2}}\right\}  , & \text{if }\alpha\beta>0,\\
\mathbb{R}\left(  \alpha,\beta\right)  , & \text{otherwise.}%
\end{array}
\right.
\]
Thus, $\operatorname*{epi}{}_{K}F^{\ast}=%
{\displaystyle\bigcup\limits_{i=1}^{4}}
N_{i},$ where%
\[%
\begin{array}
[c]{ll}%
N_{1} & =\left\{  \left(  0,0,y_{1},y_{2}\right)  :y_{1}\geq0\vee y_{2}%
\geq0\right\}  ,\\
N_{2} & =\left\{  \left(  \alpha,\beta,y_{1},y_{2}\right)  :\alpha
\beta<0\wedge y_{2}\geq\frac{\beta}{\alpha}y_{1}\right\}  ,\\
N_{3} & =\left\{  \left(  \alpha,0,y_{1},y_{2}\right)  :\alpha\neq0\wedge
y_{2}\geq0\right\}  ,\\
N_{4} & =\left\{  \left(  0,\beta,y_{1},y_{2}\right)  :\beta\neq0\wedge
y_{1}\geq0\right\}  .
\end{array}
\]
Observe that, given $\left(  \alpha,\beta,y_{1},y_{2}\right)  \in
\operatorname*{cl}N_{2},$ we have%
\[
\left\{
\begin{array}
[c]{lll}%
\alpha\beta<0 & \Longrightarrow & \left(  \alpha,\beta,y_{1},y_{2}\right)  \in
N_{2},\\
\alpha=0=\beta & \Longrightarrow & \left(  \alpha,\beta,y_{1},y_{2}\right)
\in N_{1},\\
\alpha<0=\beta & \Longrightarrow & \left(  \alpha,\beta,y_{1},y_{2}\right)
\in N_{3},\\
\alpha=0>\beta & \Longrightarrow & \left(  \alpha,\beta,y_{1},y_{2}\right)
\in N_{4},
\end{array}
\right.
\]
so that $\operatorname*{cl}N_{2}\subset\operatorname*{epi}{}_{K}F^{\ast}.$
Thus,%
\[
\operatorname*{cl}\operatorname*{epi}{}_{K}F^{\ast}\subset%
{\displaystyle\bigcup\limits_{i=1}^{4}}
\operatorname*{cl}N_{i}\subset N_{1}\cup\operatorname*{epi}{}_{K}F^{\ast}%
\cup\left(  N_{3}\cup N_{1}\right)  \cup\left(  N_{4}\cup N_{1}\right)
=\operatorname*{epi}{}_{K}F^{\ast},
\]
showing that $\operatorname*{epi}{}_{K}F^{\ast}$ is closed. However,
$\operatorname*{epi}{}_{K}F^{\ast}$ is not convex as its image by the
projection mapping $\left(  \alpha,\beta,y_{1},y_{2}\right)  \mapsto\left(
\alpha,\beta\right)  $ is the domain of $F^{\ast},$ $\operatorname{dom}%
F^{\ast}=\{\left(  \alpha,\beta\right)  \in\mathbb{R}^{2}:\ \alpha\beta
\leq0\},$ which is obviously non-convex.
\end{example}

\section{Representing $\boldsymbol{\operatorname*{epi}}$$\boldsymbol{_{K}%
(F+I_{A})^{\ast}}$}

Let $X,$ $Y,$ $Z,$ $F$ and $G$ be as in Section 2. Assume further that $F$ and
$G$ are proper mappings, $K$ is a closed, pointed, convex cone in $Y$ with
nonempty interior, and $S$ is a convex cone in $Z.$ Moreover, $C$ is a
nonempty convex subset of $X$ and $A:=C\cap G^{-1}(-S)$.

The following lemmas are useful for the representation of $\operatorname*{epi}%
{}_{K}(F+I_{A})^{\ast}$ in this section.\medskip

\begin{lemma}
\label{rem1.1} Let $F\colon X\rightarrow Y^{\bullet}$. It holds
\[
(L,y)\in\operatorname*{epi}{}_{K}F^{\ast}\quad\Longleftrightarrow\quad
y-L(x)+F(x)\notin-\operatorname*{int}K,\;\forall x\in X.
\]

\end{lemma}

\begin{proof}
It is a direct consequence of \cite[Theorem 3.1]{DGLM16} with $f=F$,
$g\equiv0_{Z}$, and $C=X$.\medskip\ 
\end{proof}

The main results in this section are two extensions of the following
characterization of the epigraph of $(f+i_{A})^{\ast}$ for a given scalar
function $f$ (recall that $\mathcal{L}_{+}(S,\mathbb{R}_{+})$ and
$\mathcal{L}_{+}^{w}(S,\mathbb{R}_{+})$ are alternative generalizations of the
dual cone $S^{+}$ to the vector setting).\medskip

\begin{lemma}
\cite[Theorem 8.2]{Bot09} \label{theo1} Let $C$ be a nonempty closed convex
subset of $X$, $f\colon X\rightarrow\mathbb{R}\cup\{+\infty\}$ be a proper
lower semicontinuous (lsc) convex function, and $G\colon X\rightarrow
Z^{\bullet}$ be a proper $S$-convex and $S$-epi closed mapping. Assume that
$A\cap\operatorname{dom}f\neq\emptyset.$ Then,
\[
\operatorname*{epi}(f+i_{A})^{\ast}=\operatorname*{cl}\left[  \bigcup
_{z^{\ast}\in S^{+}}\operatorname*{epi}(f+i_{C}+z^{\ast}\circ G)^{\ast
}\right]  .
\]

\end{lemma}

\begin{lemma}
\label{lem1} Let $F\colon X\rightarrow Y^{\bullet}$ be a proper $K$-convex
mapping and $C$ be a convex subset of $X$ such that $C\cap\operatorname{dom}%
F\neq\emptyset$. Then $F(C\cap\operatorname{dom}F)+\operatorname*{int}K$ is a
convex subset of $Y$.
\end{lemma}

\begin{proof}
Let arbitrary $y_{1},y_{2}\in F(C\cap\operatorname{dom}f)+\operatorname*{int}%
K$ and $\lambda\in\left]  0,1\right[  ,$ we will prove that $\lambda
y_{1}+(1-\lambda)y_{2}\in F(C\cap\operatorname{dom}f)+\operatorname*{int}K$.

Since $y_{1},y_{2}\in F(C\cap\operatorname{dom}f)+\operatorname*{int}K$, there
exists $x_{1},x_{2}\in C\cap\operatorname{dom}F$ such that $y_{1}\in
F(x_{1})+\operatorname*{int}K $, $y_{2}\in F(x_{2})+\operatorname*{int}K$ and
consequently,
\begin{equation}
\lambda y_{1}+(1-\lambda)y_{2}\in\lambda F(x_{1})+(1-\lambda)F(x_{2}%
)+\operatorname*{int}K.\label{0.2}%
\end{equation}

Now, because $(x_{1},F(x_{1})),(x_{2},F(x_{2}))\in\operatorname*{epi}{}_{K}F$
and $F$ is a $K$-convex mapping, one has $\lambda(x_{1},F(x_{1}))+(1-\lambda
)(x_{2},F(x_{2}))\in\operatorname*{epi}{}_{K}F$, which means
\begin{equation}
\lambda F(x_{1})+(1-\lambda)F(x_{2})\in F(\lambda x_{1}+(1-\lambda
)x_{2})+K.\label{0.3}%
\end{equation}
It follows from \eqref{0.2}, \eqref{0.3} and the equality
\begin{equation}
K+\operatorname*{int}K=\operatorname*{int}K\label{marco1}%
\end{equation}
(see e.g. \cite[(7)]{DGLM16}) that
\[
\lambda y_{1}+(1-\lambda)y_{2}\in F(\lambda x_{1}+(1-\lambda)x_{2}%
)+\operatorname*{int}K,
\]
and we are done (note that $\lambda x_{1}+(1-\lambda)x_{2}\in C\cap
\operatorname{dom}F$ since $C\cap\operatorname{dom}f$ is convex). \medskip
\end{proof}

The next lemma is proved in \cite[Lemma 1.3]{Bre11} under the assumption that
$Y$ is a normed space. \medskip

\begin{lemma}
\label{lem3} Let $M\subset Y$ be a nonempty open convex set and let $\bar
{y}\in Y$ with $\bar{y}\notin M$. Then there exists $y^{\ast}\in Y^{\ast}$
such that
\[
y^{\ast}(u)<y^{\ast}(\bar{y}),\quad\forall u\in M.
\]

\end{lemma}

\begin{proof}
It is a straightforward consequence of Theorem 3.4 in \cite{Rudin91}. \medskip
\end{proof}

\begin{lemma}
\label{lem2} Let $\bar{y}\in Y$, $y^{\ast}\in Y^{\ast}$ and $\emptyset\neq
M\subset Y,$ and assume that
\begin{equation}
y^{\ast}(u)<y^{\ast}(\bar{y}),\quad\forall u\in M-\operatorname*{int}%
K.\label{0.1}%
\end{equation}
Then, the following statements hold:\newline$(i)$\textrm{\ }$y^{\ast}(v)\leq
y^{\ast}(\bar{y}),\;\forall v\in M$;\textrm{\newline}$(ii)$\textrm{\ }%
$y^{\ast}(k)>0$ for all $k\in\operatorname*{int}K$ and, consequently,
$y^{\ast}\in K^{+}$.\textrm{\ }
\end{lemma}

\begin{proof}
$(i)$ Take $k_{0}\in\operatorname*{int}K$. Then, for any $v\in M$, it follows
from \eqref{0.1} that
\[
y^{\ast}(v-\lambda k_{0})<y^{\ast}(\bar{y}),\quad\forall\lambda>0,
\]
and by letting $\lambda\rightarrow0$, we get $y^{\ast}(v)\leq y^{\ast}(\bar
{y}).$

$(ii)$ Take arbitrarily $k\in\operatorname*{int}K$. We firstly show that there
exists $\lambda>0$ such that $\bar{y}-\lambda k\in M-\operatorname*{int}K$.
Indeed, take $m_{0}\in M$ and $k_{0}\in K$. Because of the continuity of the
mapping $t\mapsto(m_{0}-k_{0}-\bar{y})t+k$ at $t=0$, there is a $\epsilon>0 $
such that $(m_{0}-k_{0}-\bar{y})\epsilon+k\in\operatorname*{int}K$. Taking
$\lambda=\frac{1}{\epsilon},$ we obtain $m_{0}-k_{0}-\bar{y}+\lambda
k\in\lambda\operatorname*{int}K$, and consequently, applying (\ref{marco1}),
\begin{align*}
\bar{y}-\lambda k  & \in m_{0}-k_{0}-\lambda\operatorname*{int}K\subset
M-K-\operatorname*{int}K\\
& =M-\operatorname*{int}K.
\end{align*}
It now follows from \eqref{0.1} that $y^{\ast}(\bar{y}-\lambda k)<y^{\ast
}(\bar{y})$, which yields $y^{\ast}(k)>0$. Since $K=\operatorname*{cl}%
(\operatorname*{int}K)$, $y^{\ast}(k)\geq0$ for all $k\in K$ which means that
$y^{\ast}\in K^{+}$ and the proof is complete. \medskip
\end{proof}

\begin{lemma}
\label{lem4} If $F\colon X\rightarrow Y\cup\{+\infty_{Y}\}$ is a proper
mapping, then $\operatorname*{epi}{}_{K}F^{\ast}$ is a closed subset of
$\mathcal{L}(X,Y)\times Y.$
\end{lemma}

\begin{proof}
Let $\{(L_{i},y_{i})\}_{i\in I}\subset\operatorname*{epi}{}_{K}F^{\ast}$ be a
net such that $(L_{i},y_{i})\rightarrow(L,y)$. We will show that
$(L,y)\in\operatorname*{epi}{}_{K}F^{\ast}$. Let us suppose the contrary, that
is $(L,y)\notin\operatorname*{epi}{}_{K}F^{\ast}$. Then, by Lemma
\ref{rem1.1}, there exists $\bar{x}\in\operatorname{dom}F$ such that
\[
y-L(\bar{x})+F(\bar{x})\in-\operatorname*{int}K.
\]
As $y_{i}-L_{i}(\bar{x})+F(\bar{x})\rightarrow y-L(\bar{x})+F(\bar{x})$, there
is a $i_{0}\in I$ such that for all $i\in I$, $i\succsim i_{0}$, where
$\succsim$ is the net order,%
\[
y_{i}-L_{i}(\bar{x})+F(\bar{x})\in-\operatorname*{int}K,
\]
which again by Lemma \ref{rem1.1}, yields $(L_{i},y_{i})\not \in
\operatorname*{epi}{}_{K}F^{\ast}$ for all $i\succsim i_{0}$, and this is a
contradiction. \medskip
\end{proof}

\begin{theorem}
[1st asymptotic representation of $\operatorname*{epi}{}_{K}(F+I_{A})^{\ast}$%
]\label{theo2} Let $C$ be a nonempty closed convex subset of $X$, $F\colon
X\rightarrow Y\cup\{+\infty_{Y}\}$ be a proper $K$-convex mapping such that
$y^{\ast}\circ F$ is lsc for all $y^{\ast}\in Y^{\ast}$, and $G\colon
X\rightarrow Z\cup\{+\infty_{Z}\}$ be a proper $S$-convex and $S$-epi closed
mapping. Assume that $A\cap\operatorname{dom}F\neq\emptyset$. Then
\begin{equation}
\operatorname*{epi}{}_{K}(F+I_{A})^{\ast}=\operatorname*{cl}\left[
\bigcup_{T\in\mathcal{L}_{+}(S,K)}\operatorname*{epi}{}_{K}(F+I_{C}+T\circ
G)^{\ast}\right]  .\label{eqthm21a}%
\end{equation}

\end{theorem}

\begin{proof}
$\bullet$ According to \cite[Lemma 4.1]{DGLM16},
\[
\operatorname*{epi}{}_{K}(F+I_{A})^{\ast}\supset\bigcup_{T\in\mathcal{L}%
_{+}(S,K)}\operatorname*{epi}{}_{K}(F+I_{C}+T\circ G)^{\ast},
\]
which together with Lemma \ref{lem4} yields
\begin{equation}
\operatorname*{epi}{}_{K}(F+I_{A})^{\ast}\supset\operatorname*{cl}\left[
\bigcup_{T\in\mathcal{L}_{+}(S,K)}\operatorname*{epi}{}_{K}(F+I_{C}+T\circ
G)^{\ast}\right]  .\label{eqthm21a1}%
\end{equation}
$\bullet$ To prove (\ref{eqthm21a}), it is sufficient to show that the
converse inclusion in (\ref{eqthm21a1}) also holds. For this, take arbitrarily
$(L,y)\in\operatorname*{epi}{}_{K}(F+I_{A})^{\ast}$ and we will show that
\begin{equation}
(L,y)\in\operatorname*{cl}\left[  \bigcup_{T\in\mathcal{L}_{+}(S,K)}%
\operatorname*{epi}{}_{K}(F+I_{C}+T\circ G)^{\ast}\right]  .\label{eqb}%
\end{equation}

Observe that if $(L,y)\in\operatorname*{epi}{}_{K}(F+I_{A})^{\ast}$ then, by
Lemma \ref{rem1.1},
\[
y\notin L(x)-F(x)-\operatorname*{int}K,\quad\forall x\in A\cap
\operatorname{dom}F,
\]
or equivalently, $y\notin(L-F)(A\cap\operatorname{dom}F)-\operatorname*{int}K$.

$\bullet$ Now, since $G$ is $S$-convex, $G^{-1}(-S)$ is a convex set, and
hence, $A=C\cap G^{-1}(-S)$ is convex, too. Moreover, $F-L$ is a $K$-convex
mapping (as $F$ is $K$-convex), and we get from Lemma \ref{lem1} that
$(F-L)(A\cap\operatorname{dom}F)+\operatorname*{int}K$ is convex, or
equivalently, $(L-F)(A\cap\operatorname{dom}F)-\operatorname*{int}K$ is convex.

On the one hand, as $y\notin(L-F)(A\cap\operatorname{dom}%
F)-\operatorname*{int}K$, Lemma \ref{lem3} ensures the existence of $y^{\ast
}\in Y^{\ast}$ satisfying
\[
y^{\ast}(u)<y^{\ast}(y),\quad\forall u\in(L-F)(A\cap\operatorname{dom}%
F)-\operatorname*{int}K.
\]
It then follows from Lemma \ref{lem2} that
\begin{align}
& y^{\ast}\circ(L-F)(x)\leq y^{\ast}(y),\quad\forall x\in A\cap
\operatorname{dom}F,\label{1.1}\\
& y^{\ast}\in K^{+}\text{ and }\ y^{\ast}(k)>0\ \text{ }\forall k\in
\operatorname*{int}K.\label{eq15}%
\end{align}

$\bullet$ On the other hand, since $y^{\ast}\circ F$ is a proper convex lsc
function, applying Lemma \ref{theo1} to the scalar function $y^{\ast}\circ F,$
one gets
\begin{equation}
\operatorname*{epi}(y^{\ast}\circ F+i_{A})^{\ast}=\operatorname*{cl}\left[
\bigcup_{z^{\ast}\in S^{+}}\operatorname*{epi}(y^{\ast}\circ F+i_{C}+z^{\ast
}\circ G)^{\ast}\right]  .\label{2}%
\end{equation}

Note that \eqref{1.1} is equivalent to $y^{\ast}(y)\geq(y^{\ast}\circ
F+i_{A})^{\ast}(y^{\ast}\circ L)$ or, equivalently, $(y^{\ast}\circ L,y^{\ast
}(y))\in\operatorname*{epi}(y^{\ast}\circ F+i_{A})^{\ast}$. Hence, by
\eqref{2}, there exist nets $\{z_{i}^{\ast}\}_{i\in I}\subset S^{+}$,
$\{x_{i}^{\ast}\}_{i\in I}\subset X^{\ast}$ and $\{r_{i}\}_{i\in I}%
\subset\mathbb{R}$ such that $x_{i}^{\ast}\rightarrow y^{\ast}\circ L$,
$r_{i}\rightarrow y^{\ast}(y)$ and
\begin{equation}
(x_{i}^{\ast},r_{i})\in\operatorname*{epi}(y^{\ast}\circ F+i_{C}+z_{i}^{\ast
}\circ G)^{\ast},\quad\forall i\in I.\label{7bis}%
\end{equation}
Take an arbitrary $k_{0}\in\operatorname*{int}K$. Then $y^{\ast}(k_{0})>0$
(see (\ref{eq15})).

Now for each $i\in I$, set
\[
y_{i}:=y+\frac{r_{i}-y^{\ast}(y)}{y^{\ast}(k_{0})}k_{0},
\]
and define the mapping $L_{i}:X\longrightarrow Y$ by
\[
L_{i}(x):=L(x)+\frac{x_{i}^{\ast}(x)-\left(  y^{\ast}\circ L\right)
(x)}{y^{\ast}(k_{0})}k_{0},\ \forall x\in X.
\]
It is easy to check that%
\begin{equation}
y^{\ast}(y_{i})=r_{i},\ L_{i}\in\mathcal{L}(X,Y),\ y^{\ast}\circ L_{i}%
=x_{i}^{\ast},\forall i\in I\ \text{ and }(y_{i},L_{i})\rightarrow
(y,L).\label{2.1}%
\end{equation}

$\bullet$ We now claim that
\[
(L_{i},y_{i})\in\bigcup_{T\in\mathcal{L}_{+}(S,K)}\operatorname*{epi}{}%
_{K}(F+I_{C}+T\circ G)^{\ast},\quad\forall i\in I.
\]
Indeed, for each $i\in I$, combining \eqref{7bis} and (\ref{2.1}) we get
\[
y^{\ast}(y_{i})\geq(y^{\ast}\circ F+i_{C}+z_{i}^{\ast}\circ G)^{\ast}(y^{\ast
}\circ L_{i}),
\]
or equivalently,
\begin{equation}
y^{\ast}(y_{i})\geq\left(  y^{\ast}\circ L_{i}\right)  (x)-\left(  y^{\ast
}\circ F\right)  (x)-\left(  z_{i}^{\ast}\circ G\right)  (x),\quad\forall x\in
C\cap\operatorname{dom}F.\label{8}%
\end{equation}
For each $i\in I$, define $T_{i}\colon Z\longrightarrow Y$ by
\[
T_{i}(z):=\frac{z_{i}^{\ast}(z)}{y^{\ast}(k_{0})}k_{0},\ \forall z\in Z.
\]
Then $T_{i}\in\mathcal{L}(Z,Y).$ Moreover, if $z\in S,$ then $z_{i}^{\ast
}(z)\geq0$ (as $z_{i}^{\ast}\in S^{+}$) and so, $T_{i}(z)\in K$ (as $k_{0}%
\in\operatorname*{int}K$ and $y^{\ast}(k_{0})>0$). Consequently, $T_{i}%
\in\mathcal{L}_{+}(S,K)$.

Since $y^{\ast}\circ T_{i}=z_{i}^{\ast}$, with the help of the mappings
$T_{i}\in\mathcal{L}_{+}(S,K)$, $i\in I$, \eqref{8} can be rewritten as
\[
y^{\ast}(y_{i})\geq\left(  y^{\ast}\circ L_{i}\right)  (x)-\left(  y^{\ast
}\circ F\right)  (x)-(y^{\ast}\circ T_{i}\circ G)(x),\quad\forall x\in
C\cap\operatorname{dom}F,
\]
or equivalently,
\[
y^{\ast}\left(  L_{i}(x)-\ F(x)-(T_{i}\circ G)(x)-y_{i}\right)  \leq
0,\ \forall x\in C\cap\operatorname{dom}F.
\]
The last inequality, together with \eqref{eq15}, implies that
\[
y_{i}\notin L_{i}(x)-F(x)-\left(  T_{i}\circ G\right)  (x)-\operatorname*{int}%
K,\quad\forall x\in C\cap\operatorname{dom}F,
\]
which, together with Lemma \ref{rem1.1}, yields $(L_{i},y_{i})\in
\operatorname*{epi}{}_{K}(F+I_{C}+T_{i}\circ G)^{\ast}$.

Finally, taking (\ref{2.1}) into account, \eqref{eqb} follows and we are done.
\medskip
\end{proof}

We now show that the closure in the right-hand side of \eqref{eqthm21a} in
Theorem \ref{theo2} can be removed if certain qualification condition holds.
To do this we need the lemma below on scalar functions.\medskip

\begin{lemma}
\label{lemscalar2} Let $C$ be a nonempty convex subset of $X$, $f\colon
X\rightarrow\mathbb{R}\cup\{+\infty\}$ be a proper convex function, and
$G\colon\ X\rightarrow Z\cup\{+\infty_{Z}\}$ be a proper $S$-convex mapping.
Let $D:=G\left(  C\cap\operatorname{dom}f\cap\operatorname{dom}G\right)
+S.$\ Assume that $A\cap\operatorname{dom}f\neq\emptyset$ and one of the
following conditions is fulfilled:\newline$(i)$ There exists $\bar{x}\in
C\cap\operatorname{dom}f$ such that $G(\bar{x})\in-\operatorname*{int}%
S;\newline(ii)$ $X,Z$ are Fr\'{e}chet spaces, $C$ is closed, $f$ is lsc, $G$
is $S$-epi closed and $0_{Z}\in\operatorname*{sqri}D;$\newline$(iii)$
$\dim\operatorname*{lin}D<+\infty$ and\thinspace\ $0_{Z}\in\operatorname*{ri}%
D.$\textrm{\ }
\end{lemma}

Then
\[
\operatorname*{epi}(f+i_{A})^{\ast}=\bigcup_{z^{\ast}\in S^{+}}%
\operatorname*{epi}(f+i_{C}+z^{\ast}\circ G)^{\ast}.
\]

\begin{proof}
Take an arbitrary $x^{\ast}\in X^{\ast}$. Applying \cite[Theorem 3.4]{Bot09},
with $f-x^{\ast}$ playing the role of primal objective function, we get the
existence of $\bar{z}^{\ast}\in S^{+}$ satisfying%
\[%
\begin{array}
[c]{ll}%
\inf\limits_{_{\substack{x\in C \\g(x)\in-S}}}[f(x)-x^{\ast}(x)] &
=\max\limits_{z^{\ast}\in S^{+}}\inf\limits_{x\in C}[f(x)-x^{\ast}(x)+\left(
z^{\ast}\circ G\right)  (x)]\\
& =\inf\limits_{x\in C}[f(x)-x^{\ast}(x)+\left(  \bar{z}^{\ast}\circ G\right)
(x)],
\end{array}
\]
which is equivalent to
\[
(f+i_{A})^{\ast}(x^{\ast})=(f+i_{C}+\bar{z}^{\ast}\circ G)^{\ast}(x^{\ast}),
\]
for some $\bar{z}^{\ast}\in S^{+}$, showing that
\[
\operatorname*{epi}(f+i_{A})^{\ast}=\bigcup_{z^{\ast}\in S^{+}}%
\operatorname*{epi}(f+i_{C}+z^{\ast}\circ G)^{\ast}%
\]
and we are done. \medskip
\end{proof}

\begin{theorem}
[1st non-asymptotic representation of $\operatorname*{epi}{}_{K}%
(F+I_{A})^{\ast}$]\label{theovector2} Let $C$ be a nonempty convex subset of
$X$, $F\colon X\rightarrow Y\cup\{+\infty_{Y}\}$ be a proper $K$-convex
mapping, and $G\colon\ X\rightarrow Z\cup\{+\infty_{Z}\}$ be a proper
$S$-convex mapping. Consider the set $E:=G\left(  C\cap\operatorname{dom}%
F\cap\operatorname{dom}G\right)  +S.$ Assume that $A\cap\operatorname{dom}%
F\neq\emptyset$ and at least one of the following qualification conditions
holds:\newline$(c_{1})$ There exists $\bar{x}\in C\cap\operatorname*{dom}F$
such that $G(\bar{x})\in-\operatorname*{int}S;$\newline$(c_{2})$ $X,Z$ are
Fr\'{e}chet spaces, $C$ is closed, $y^{\ast}\circ F$ is lsc for all $y^{\ast
}\in Y^{\ast}$, $G$ is $S$-epi closed and $0_{Z}\in\operatorname*{sqri}%
E;$\newline$(c_{3})$ $\dim\operatorname*{lin}E<+\infty$ and $0_{Z}%
\in\operatorname*{ri}E$.\newline Then,
\[
\operatorname*{epi}{}_{K}(F+I_{A})^{\ast}=\bigcup_{T\in\mathcal{L}_{+}%
(S,K)}\operatorname*{epi}{}_{K}(F+I_{C}+T\circ G)^{\ast}.
\]

\end{theorem}

\begin{proof}
The proof goes in parallel to the one of Theorem \ref{theo2}, using Lemma
\ref{lemscalar2} instead of Lemma \ref{theo1}. For an easy reading, the main
ideas will be repeated below.

$\bullet$ By \cite[Lemma 4.1]{DGLM16}, it is sufficient to show that
\begin{equation}
\operatorname*{epi}{}_{K}(F+I_{A})^{\ast}\subset\bigcup_{T\in\mathcal{L}%
_{+}(S,K)}\operatorname*{epi}{}_{K}(F+I_{C}+T\circ G)^{\ast}.\label{eqe}%
\end{equation}

$\bullet$ Take an arbitrary $(L,y)\in\operatorname*{epi}{}_{K}(F+I_{A})^{\ast
}$. Then, by the same argument as the one in the proof of Theorem \ref{theo2},
using Lemmas \ref{rem1.1}, \ref{lem1}, \ref{lem3}, and \ref{lem2}, there
exists $y^{\ast}\in Y^{\ast}$ such that \eqref{1.1} and \eqref{eq15} hold.

Observe also that \eqref{1.1} is equivalent to $y^{\ast}(y)\geq(y^{\ast}\circ
F+i_{A})^{\ast}(y^{\ast}\circ L)$, which accounts for
\begin{equation}
(y^{\ast}\circ L,y^{\ast}(y))\in\operatorname*{epi}(y^{\ast}\circ
F+i_{A})^{\ast}.\label{eqc}%
\end{equation}

$\bullet$ Because $y^{\ast}\in K^{+}$ and $F$ is a $K$-convex mapping,
$y^{\ast}\circ F$ is a convex function. If one of the qualification conditions
$(c_{1}),(c_{2}),(c_{3})$ holds then, by Lemma \ref{lemscalar2}, one has
\begin{equation}
\operatorname*{epi}(y^{\ast}\circ F+i_{A})^{\ast}=\bigcup_{z^{\ast}\in S^{+}%
}\operatorname*{epi}(y^{\ast}\circ F+i_{C}+z^{\ast}\circ G)^{\ast
}.\label{2bis}%
\end{equation}
This and \eqref{eqc} ensure the existence of $z^{\ast}\in S^{+}$ satisfying
$(y^{\ast}\circ L,y^{\ast}(y))\in\operatorname*{epi}(y^{\ast}\circ
F+i_{C}+z^{\ast}\circ G)^{\ast}$, which means that
\begin{equation}
y^{\ast}(y)\geq\left(  y^{\ast}\circ L\right)  (x)-\left(  y^{\ast}\circ
F\right)  (x)-\left(  z^{\ast}\circ G\right)  (x),\quad\forall x\in
C\cap\operatorname{dom}F.\label{eqd}%
\end{equation}

$\bullet$ Now, pick $k_{0}\in\operatorname*{int}K$ and consider the linear
mapping $T\colon Z\longrightarrow Y$ such that
\[
T(z):=\frac{z^{\ast}(z)}{y^{\ast}(k_{0})}k_{0},\forall z\in Z.
\]
Then $T\in\mathcal{L}_{+}(S,K)$ and $y^{\ast}\circ T=z^{\ast}$. Hence,
\eqref{eqd} can be rewritten as
\[
y^{\ast}(y)\geq\left(  y^{\ast}\circ L\right)  (x)-\left(  y^{\ast}\circ
F\right)  (x)-(y^{\ast}\circ T\circ G)(x),\quad\forall x\in C\cap
\operatorname{dom}F,
\]
or equivalently,
\[
y^{\ast}\left(  L(x)-F(x)-\left(  T\circ G\right)  (x)-y\right)  \leq
0,\quad\forall x\in C\cap\operatorname{dom}F.
\]
So, by \eqref{eq15},
\[
L(x)-F(x)-\left(  T\circ G\right)  (x)-y\notin\operatorname*{int}%
K,\quad\forall x\in C\cap\operatorname{dom}F,
\]
which in turn yields, by Lemma \ref{rem1.1}, $(L,y)\in\operatorname*{epi}%
{}_{K}(F+I_{C}+T\circ G)^{\ast}$. Hence, \eqref{eqe} has been proved and the
proof is complete. \medskip
\end{proof}

\begin{example}
\label{Example2}Let $X,$ $Y,$ $Z,$ $F,$ and $G$ be as in Example
\ref{Example1}.\textbf{\ }Let $C=\mathbb{R}.$\ Due to the extreme simplicity
of $A=C\cap G^{-1}(-S)=\mathbb{R}_{+}$ in this case, $\operatorname*{epi}%
{}_{K}(F+I_{A})^{\ast}$ can be calculated directly. In fact, since
$(F+I_{A})^{\ast}\left(  \alpha,\beta\right)  =\operatorname*{WSup}%
\{\mathbb{R}_{+}(\alpha,\beta)\},$ one gets
\[
(F+I_{A})^{\ast}\left(  \alpha,\beta\right)  =\left\{
\begin{array}
[c]{ll}%
\left\{  +\infty_{\mathbb{R}^{2}}\right\}  , & \text{if }\alpha>0\text{ and
}\beta>0,\\
\left[  (-\mathbb{R}_{+})\times\{0\}\right]  \cup\left[  \{0\}\times
(-\mathbb{R}_{+})\right]  , & \text{if }\alpha\leq0\text{ and }\beta\leq0,\\
\mathbb{R}\left(  \alpha,\beta\right)  , & \text{if }\alpha\beta=0\text{ and
}\alpha+\beta>0,\\
\mathbb{R}_{+}\left(  \alpha,\beta\right)  \cup\left[  (-\mathbb{R}_{+}%
)\times\{0\}\right]  , & \text{if }\alpha>0\text{ and }\beta<0,\\
\mathbb{R}_{+}\left(  \alpha,\beta\right)  \cup\left[  \{0\}\times
(-\mathbb{R}_{+})\right]  , & \text{if }\alpha<0\text{ and }\beta>0.
\end{array}
\right.
\]
Thus, $\operatorname*{epi}{}_{K}(F+I_{A})^{\ast}=%
{\displaystyle\bigcup\limits_{i=1}^{5}}
P_{i},$ where%
\[%
\begin{array}
[c]{ll}%
P_{1} & =\left\{  \left(  \alpha,\beta,y_{1},y_{2}\right)  :\alpha\leq
0\wedge\beta\leq0\wedge\left(  y_{1}\geq0\vee y_{2}\geq0\right)  \right\}  ,\\
P_{2} & =\left\{  \left(  0,\beta,y_{1},y_{2}\right)  :\beta>0\wedge y_{1}%
\geq0\right\}  ,\\
P_{3} & =\left\{  \left(  \alpha,0,y_{1},y_{2}\right)  :\alpha>0\wedge
y_{2}\geq0\right\}  ,\\
P_{4} & =\left\{  \left(  \alpha,\beta,y_{1},y_{2}\right)  :\alpha
>0\wedge\beta<0\wedge y_{2}\geq\min\left\{  0,\frac{\beta}{\alpha}%
y_{1}\right\}  \right\}  ,\smallskip\\
P_{5} & =\left\{  \left(  \alpha,\beta,y_{1},y_{2}\right)  :\alpha
<0\wedge\beta>0\wedge y_{1}\geq\min\left\{  0,\frac{\alpha}{\beta}%
y_{2}\right\}  \right\}  .
\end{array}
\]
Since $\operatorname*{dom}(F+I_{A})^{\ast}=\left\{  \left(  \alpha
,\beta\right)  \in\mathbb{R}^{2}:\alpha\leq0\vee\beta\leq0\right\}  $ is not
convex, $\operatorname*{epi}{}_{K}(F+I_{A})^{\ast}$ cannot be convex while its
closedness follows from Lemma \ref{lem4} applied to the proper vector function
$F+I_{A}=I_{A}.$\newline According to Theorem \ref{theovector2}, as the
interior type condition ($c_{1}$) is satisfied by any positive number, we can
also express
\[
\operatorname*{epi}{}_{K}(F+I_{A})^{\ast}=\bigcup_{\left(  t_{1},t_{2}\right)
\in\mathbb{R}_{+}^{2}}\operatorname*{epi}{}_{K}(\left(  t_{1},t_{2}\right)
\circ G)^{\ast},
\]
where
\[
(\left(  t_{1},t_{2}\right)  \circ G)^{\ast}\left(  \alpha,\beta\right)
=\operatorname*{WSup}\{\mathbb{R}\left(  \alpha+t_{1},\beta+t_{2}\right)
\}=F^{\ast}\left(  \alpha+t_{1},\beta+t_{2}\right)  .
\]
So, $\operatorname*{epi}{}_{K}(\left(  t_{1},t_{2}\right)  \circ G)^{\ast}=$ $%
{\displaystyle\bigcup\limits_{i=1}^{4}}
Q_{i}\left(  t_{1},t_{2}\right)  ,$ with%
\[%
\begin{array}
[c]{ll}%
Q_{1}\left(  t_{1},t_{2}\right)  & =\left\{  \left(  -t_{1},-t_{2},y_{1}%
,y_{2}\right)  :y_{1}\geq0\vee y_{2}\geq0\right\}  ,\\
Q_{2}\left(  t_{1},t_{2}\right)  & =\left\{  \left(  \alpha,\beta,y_{1}%
,y_{2}\right)  :\left(  \alpha+t_{1}\right)  \left(  \beta+t_{2}\right)
<0\wedge y_{2}\geq\left(  \frac{\beta+t_{2}}{\alpha+t_{1}}\right)
y_{1}\right\}  ,\\
Q_{3}\left(  t_{1},t_{2}\right)  & =\left\{  \left(  \alpha,-t_{2},y_{1}%
,y_{2}\right)  :\alpha\neq-t_{1}\wedge y_{2}\geq0\right\}  ,\\
Q_{4}\left(  t_{1},t_{2}\right)  & =\left\{  \left(  -t_{1},\beta,y_{1}%
,y_{2}\right)  :\beta\neq-t_{2}\wedge y_{1}\geq0\right\}  .
\end{array}
\]
From Theorem \ref{theovector2} and the inclusion $\mathcal{L}_{+}%
(S,K)\subset\mathcal{L}_{+}^{w}(S,K),$ one has
\[
\operatorname*{epi}{}_{K}(F+I_{A})^{\ast}\subset\bigcup_{T\in\mathcal{L}%
_{+}^{w}(S,K)}\operatorname*{epi}{}_{K}(F+I_{C}+T\circ G)^{\ast}.
\]
Next we show that this inclusion might be strict under the assumptions of
Theorem \ref{theovector2}. Indeed,%
\[%
\begin{array}
[c]{ll}%
\left(  1,0,0,-1\right)  & \in Q_{1}\left(  -1,0\right)  \diagdown\left(
{\displaystyle\bigcup\limits_{i=1}^{5}}
P_{i},\right) \\
& \subset\left[  \bigcup\limits_{T\in\mathcal{L}_{+}^{w}(S,K)}%
\operatorname*{epi}{}_{K}(F+I_{C}+T\circ G)^{\ast}\right]  \diagdown
\operatorname*{epi}{}_{K}(F+I_{A})^{\ast}.
\end{array}
\]

\end{example}

The rest of this section is devoted to derive representations of
$\operatorname*{epi}{}_{K}(F+I_{A})^{\ast}$ where the set
\[
\mathcal{L}_{+}^{w}(S,K)=\{T\in\mathcal{L}(X,Y):T(S)\cap(-\operatorname*{int}%
K)=\emptyset\}
\]
replaces $\mathcal{L}_{+}(S,K)$ as index set at the right-hand side union of
sets.\medskip

\begin{lemma}
\label{lem2.1} One has
\begin{equation}
\operatorname*{epi}{}_{K}(F+I_{A})^{\ast}\supset\bigcap_{v\in I_{-S}^{\ast
}(T)}\big[\operatorname*{epi}{}_{K}(F+I_{C}+T\circ G)^{\ast}%
+(0\mathbf{{_{\mathcal{L}}}},v)\big],\quad\forall T\in\mathcal{L}_{+}%
^{w}(S,K).\label{eq3}%
\end{equation}

\end{lemma}

\begin{proof}
Take arbitrarily $T\in\mathcal{L}_{+}^{w}(S,K)$ and
\[
(L,y)\in\bigcap\limits_{v\in I_{-S}^{\ast}(T)}\big[\operatorname*{epi}{}%
_{K}(F+I_{C}+T\circ G)^{\ast}+(0\mathbf{{_{\mathcal{L}}}},v)\big].
\]
Then,
\[
(L,y-v)\in\operatorname*{epi}{}_{K}(F+I_{C}+T\circ G)^{\ast},\quad\forall v\in
I_{-S}^{\ast}(T),
\]
and, by Lemma \ref{rem1.1}, (\ref{2.3}) and \eqref{1.3}, the last inclusion is
equivalent to
\begin{align}
&  y-v-L(x)+F(x)+\left(  T\circ G\right)  (x)\notin-\operatorname*{int}%
K,\quad\forall x\in C,\;\forall v\in I_{-S}^{\ast}(T)\nonumber\\
\Longleftrightarrow\; &  y-L(x)+F(x)+\left(  T\circ G\right)  (x)\notin
I_{-S}^{\ast}(T)-\operatorname*{int}K,\quad\forall x\in C\nonumber\\
\Longleftrightarrow\; &  y-L(x)+F(x)+\left(  T\circ G\right)  (x)\notin
\operatorname*{WSup}T(-S)-\operatorname*{int}K,\quad\forall x\in C\nonumber\\
{\Longleftrightarrow}\; &  y-L(x)+F(x)+\left(  T\circ G\right)  (x)\notin
T(-S)-\operatorname*{int}K,\quad\forall x\in C\qquad\nonumber\\
\Longleftrightarrow\; &  y-L(x)+F(x)\notin u-\left(  T\circ G\right)
(x)-\operatorname*{int}K,\quad\forall u\in T(-S),\;\forall x\in C.\label{19}%
\end{align}
Now, for any $x\in A$, taking $u=\left(  T\circ G\right)  (x)$ in \eqref{19}
(note that $x\in A$ yields $G(x)\in-S$), we get $y-L(x)+F(x)\notin
-\operatorname*{int}K$. Hence, again by Lemma \ref{rem1.1}, we obtain
$(L,y)\in\operatorname*{epi}{}_{K}(F+I_{A})^{\ast}$ and \eqref{eq3} follows.
\medskip
\end{proof}

\begin{lemma}
\label{lem2.3} If $T\in\mathcal{L}_{+}(S,K)$ then
\begin{equation}
\operatorname*{epi}{}_{K}(F+I_{C}+T\circ G)^{\ast}=\bigcap_{v\in I_{-S}^{\ast
}(T)}[\operatorname*{epi}{}_{K}(F+I_{C}+T\circ G)^{\ast}%
+(0\mathbf{{_{\mathcal{L}}}},v)].\label{pt29}%
\end{equation}

\end{lemma}

\begin{proof}
Assume that $T\in\mathcal{L}_{+}(S,K)$. One has $T(-S)\subset-K$ and $0_{Y}\in
T(-S)$ (as $0_{Y}=T(0_{X})$). So, by Definition \ref{def1} and (\ref{2.3}),
$0_{Y}\in\operatorname*{WSup}T(-S)=I_{-S}^{\ast}(T)$. Hence,
$\operatorname*{epi}{}_{K}(F+I_{C}+T\circ G)^{\ast}$ is a member of the
collection in the right-hand side of \eqref{pt29}, and we obtain
\[
\bigcap_{v\in I_{-S}^{\ast}(T)}\Big[\operatorname*{epi}{}_{K}(F+I_{C}+T\circ
G)^{\ast}+(0\mathbf{{_{\mathcal{L}}}},v)\Big]\subset\operatorname*{epi}{}%
_{K}(F+I_{C}+T\circ G)^{\ast}.
\]

Conversely, take an arbitrary $(L,y)\in\operatorname*{epi}{}_{K}%
(F+I_{C}+T\circ G)^{\ast}.$ We will prove that $(L,y)\in\bigcap\limits_{v\in
I_{-S}^{\ast}(T)}[\operatorname*{epi}{}_{K}(F+I_{C}+T\circ G)^{\ast
}+(0\mathbf{{_{\mathcal{L}}}},v)]$, or equivalently,
\[
(L,y)\in\operatorname*{epi}{}_{K}(F+I_{C}+T\circ G)^{\ast}%
+(0\mathbf{{_{\mathcal{L}}}},v),\quad\forall v\in I_{-S}^{\ast}(T).
\]
Since $T(-S)\subset-K$, it follows from \eqref{formularwsup} that
$\operatorname*{WSup}T(-S)\subset\operatorname*{cl}\left[
T(-S)-\operatorname*{int}K\right]  \subset-K$, and consequently $I_{-S}^{\ast
}(T)\subset-K$.

Since $(L,y)\in\operatorname*{epi}{}_{K}(F+I_{C}+T\circ G)^{\ast}$, one has
$y\in(F+I_{C}+T\circ G)^{\ast}(L)+K$. For any $v\in I_{-S}^{\ast}(T)$, as
$I_{-S}^{\ast}(T)\subset-K$, we get
\[
y-v\in(F+I_{C}+T\circ G)^{\ast}(L)+K+K=(F+I_{C}+T\circ G)^{\ast}(L)+K,
\]
which accounts for $(L,y)\in\operatorname*{epi}{}_{K}(F+I_{C}+T\circ G)^{\ast
}(L)+(0\mathbf{{_{\mathcal{L}}}},v)$ and we are done. \medskip
\end{proof}

\begin{proposition}
\label{pro2.1} One has
\begin{align*}
\operatorname*{epi}{}_{K}(F+I_{A})^{\ast}  & \supset\bigcup_{T\in
\mathcal{L}_{+}^{w}(S,K)}\left[  \bigcap_{v\in I_{-S}^{\ast}(T)}%
[\operatorname*{epi}{}_{K}(F+I_{C}+T\circ G)^{\ast}+(0\mathbf{{_{\mathcal{L}}%
}},v)]\right] \\
& \supset\bigcup_{T\in\mathcal{L}_{+}(S,K)}\operatorname*{epi}{}_{K}%
(F+I_{C}+T\circ G)^{\ast}.
\end{align*}

\end{proposition}

\begin{proof}
The first inclusion follows from Lemma \ref{lem2.1} while the second one
follows from the fact that $\mathcal{L}_{+}(S,K)\subset\mathcal{L}_{+}%
^{w}(S,K)$ and Lemma \ref{lem2.3}. \medskip
\end{proof}

\begin{theorem}
[2nd asymptotic representation of $\operatorname*{epi}{}_{K}(F+I_{A})^{\ast}$%
]\label{theovectorw} Assume all the assumptions of Theorem \ref{theo2} hold.
Then,
\[
\operatorname*{epi}{}_{K}(F+I_{A})^{\ast}=\operatorname*{cl}\left\{
\bigcup_{T\in\mathcal{L}_{+}^{w}(S,K)}\left[  \bigcap_{v\in I_{-S}^{\ast}%
(T)}[\operatorname*{epi}{}_{K}(F+I_{C}+T\circ G)^{\ast}%
+(0\mathbf{{_{\mathcal{L}}}},v)]\right]  \right\}  .
\]

\end{theorem}

\begin{proof}
It follows from Lemma \ref{lem4} and Proposition \ref{pro2.1} that
\begin{align*}
\operatorname*{epi}{}_{K}(F+I_{A})^{\ast}  & \supset\operatorname*{cl}\left\{
\bigcup_{T\in\mathcal{L}_{+}^{w}(S,K)}\left[  \bigcap_{v\in I_{-S}^{\ast}%
(T)}[\operatorname*{epi}{}_{K}(F+I_{C}+T\circ G)^{\ast}%
+(0\mathbf{{_{\mathcal{L}}}},v)]\right]  \right\} \\
& \supset\operatorname*{cl}\left\{  \bigcup_{T\in\mathcal{L}_{+}%
(S,K)}\operatorname*{epi}{}_{K}(F+I_{C}+T\circ G)^{\ast}\right\}  ,
\end{align*}
and the conclusion now follows from Theorem \ref{theo2}. \medskip
\end{proof}

\begin{theorem}
[2nd non-asymptotic representation of $\operatorname*{epi}{}_{K}%
(F+I_{A})^{\ast}$]\label{theovectorw2} Assume all the assumptions of Theorem
\ref{theovector2}. Then
\[
\operatorname*{epi}{}_{K}(F+I_{A})^{\ast}=\bigcup_{T\in\mathcal{L}_{+}%
^{w}(S,K)}\left[  \bigcap_{v\in I_{-S}^{\ast}(T)}[\operatorname*{epi}{}%
_{K}(F+I_{C}+T\circ G)^{\ast}+(0\mathbf{{_{\mathcal{L}}}},v)]\right]  .
\]

\end{theorem}

\begin{proof}
By Proposition \ref{pro2.1},
\begin{align*}
\operatorname*{epi}{}_{K}(F+I_{A})^{\ast}  & \supset\bigcup_{T\in
\mathcal{L}_{+}^{w}(S,K)}\left[  \bigcap_{v\in I_{-S}^{\ast}(T)}%
[\operatorname*{epi}{}_{K}(F+I_{C}+T\circ G)^{\ast}+(0\mathbf{{_{\mathcal{L}}%
}},v)]\right] \\
\  & \supset\bigcup_{T\in\mathcal{L}_{+}(S,K)}\operatorname*{epi}{}%
_{K}(F+I_{C}+T\circ G)^{\ast}.
\end{align*}
The conclusion now follows from this double inclusion and Theorem
\ref{theovector2}. \medskip
\end{proof}

\section{Farkas-type results for vector-valued functions}

Let $X,$ $Y,$ $Z,$ $F,$ $G,$ $C,$ and $A$ be as in Section 3. We also assume
that $A\cap\operatorname*{dom}F\not =\emptyset.$

This section provides stable reverse Farkas-type results in the sense of
\cite{DGLM16} for the \emph{constraint system }of\emph{\ }\textrm{(VOP):}
\[
\{x\in C,\;G(x)\in-S\}\equiv\{x\in C,\ G(x)\leqq_{S}0_{Z}\}.
\]

We first recall a general result which will be useful in the sequel.\medskip

\begin{lemma}
\cite[Theorem 4.5]{DGLM16} \label{lemFV} The following statements are
equivalent:\medskip\newline$(a)$ $\operatorname*{epi}_{K}(F+I_{A})^{\ast
}=\bigcup\limits_{T\in\mathcal{L}_{+}(S,K)}\operatorname*{epi}_{K}%
(F+I_{C}+T\circ G)^{\ast};$\newline$(b)$ For any $y\in Y$ and any
$L\in\mathcal{L}(X,Y)$, the following assertions are equivalent:%
\[%
\begin{array}
[c]{l}%
(b_{1})\text{ }G(x)\in-S,\;x\in C\;\Longrightarrow\;F(x)-L(x)+y\notin
-\operatorname*{int}K;\\
(b_{2})\text{ }\exists T\in\mathcal{L}_{+}(S,K)\text{ such that }F(x)+(T\circ
G)(x)-L(x)+y\notin-\operatorname*{int}K\;\forall x\in C.
\end{array}
\]
$\qquad$\newline
\end{lemma}

\begin{theorem}
[1st characterization of Farkas lemma]\label{theoFV1} Let $C$ be a nonempty
closed convex subset of $X$, $F\colon\ X\rightarrow Y\cup\{+\infty_{Y}\}$ be a
proper $K$-convex mapping satisfying $y^{\ast}\circ F$ is lsc for all
$y^{\ast}\in Y^{\ast}$, and $G\colon\ X\rightarrow Z\cup\{+\infty_{Z}\}$ be a
proper $S$-convex and $S$-epi closed mapping. Then, the following statements
are equivalent:\newline$(a^{\prime})$ The set
\[
\bigcup\limits_{T\in\mathcal{L}_{+}(S,K)}\operatorname*{epi}\nolimits_{K}%
(F+I_{C}+T\circ G)^{\ast}%
\]
is closed in $\mathcal{L}(X,Y)\times Y$.\newline$(b)$ $\forall(L,y)\in
\mathcal{L}(X,Y)\times Y,\;(b_{1})\Longleftrightarrow(b_{2})$.
\end{theorem}

\begin{proof}
It follows from Theorem \ref{theo2} that
\[
\operatorname*{epi}{}_{K}(F+I_{A})^{\ast}=\operatorname*{cl}\left[
\bigcup_{T\in\mathcal{L}_{+}(S,K)}\operatorname*{epi}{}_{K}(F+I_{C}+T\circ
G)^{\ast}\right]  .
\]
Hence, (a') is equivalent to $\operatorname*{epi}{}_{K}(F+I_{A})^{\ast
}=\bigcup\limits_{T\in\mathcal{L}_{+}(S,K)}\operatorname*{epi}{}_{K}%
(F+I_{C}+T\circ G)^{\ast}$ and the conclusion follows from Lemma \ref{lemFV}.
\medskip
\end{proof}

\begin{theorem}
[1st Farkas lemma]\label{theoFV2} Let $C$ be a nonempty convex subset of $X$,
$F\colon\ X\rightarrow Y\cup\{+\infty_{Y}\}$ be a proper $K$-convex mapping,
and $G\colon\ X\rightarrow Z\cup\{+\infty_{Z}\}$ be a proper $S$-convex
mapping. Assume that $A\cap\operatorname{dom}F\neq\emptyset$ and that one of
the conditions $(c_{1})$, $(c_{2})$ and $(c_{3})$ holds. Then for all
$(L,y)\in\mathcal{L}(X,Y)\times Y$, one has $(b_{1})\Longleftrightarrow
(b_{2})\mathrm{.}$
\end{theorem}

\begin{proof}
We get from Theorem \ref{theovector2} that
\[
\operatorname*{epi}{}_{K}(F+I_{A})^{\ast}=\bigcup\limits_{T\in\mathcal{L}%
_{+}(S,K)}\operatorname*{epi}{}_{K}(F+I_{C}+T\circ G)^{\ast}%
\]
and hence, the conclusion also follows from Lemma \ref{lemFV}. \medskip
\end{proof}

\begin{theorem}
[2nd characterization of Farkas lemma]\label{lemFVw} Assume all the
assumptions of Theorem \ref{theovector2}. Then the following statements are
equivalent:\medskip\newline$(c)$ $\operatorname*{epi}_{K}(F+I_{A})^{\ast
}=\bigcup\limits_{T\in\mathcal{L}_{+}^{w}(S,K)}\left[  \bigcap\limits_{v\in
I_{-S}^{\ast}(T)}[\operatorname*{epi}_{K}(F+I_{C}+T\circ G)^{\ast
}+(0_{\mathcal{L}},v)]\right]  .$\newline$(d)$ For any $y\in Y$ and any
$L\in\mathcal{L}(X,Y)$, the following assertions are equivalent:%
\[%
\begin{array}
[c]{l}%
(b_{1})\text{ }G(x)\in-S,\;x\in C\;\Longrightarrow\;F(x)-L(x)+y\notin
-\operatorname*{int}K;\\
(b_{3})\text{ }\exists T\in\mathcal{L}_{+}^{w}(S,K):F(x)+\left(  T\circ
G\right)  (x)-L(x)+y\notin T(-S)-\operatorname*{int}K,\;\forall x\in C.
\end{array}
\]
$\ \ \ \ \ \ \ \ \ \ $\newline
\end{theorem}

\begin{proof}
$[(c)\Longrightarrow(d)]$ Take an arbitrary $(L,y)\in\mathcal{L}(X,Y)\times
Y$. On the one hand, by Lemma \ref{rem1.1}, one has%
\begin{align*}
&  (L,y)\in\operatorname*{epi}{}_{K}(F+I_{A})^{\ast}\\
&  \Longleftrightarrow\;y-L(x)+F(x)+I_{A}(x)\notin-\operatorname*{int}%
K,\;\forall x\in X\quad\\
&  \Longleftrightarrow\;y-L(x)+F(x)\notin-\operatorname*{int}K,\;\forall x\in
A.
\end{align*}
On the other hand, by an argument similar to that of (\ref{19}),
\begin{align*}
&  (L,y)\in\bigcup_{T\in\mathcal{L}_{+}^{w}(S,K)}\left[  \bigcap_{v\in
I_{-S}^{\ast}(T)}[\operatorname*{epi}{}_{K}(F+I_{C}+T\circ G)^{\ast
}+(0\mathbf{{_{\mathcal{L}}}},v)]\right] \\
&  \Longleftrightarrow\exists T\in\mathcal{L}_{+}^{w}(S,K):(L,y)\in
\bigcap_{v\in I_{-S}^{\ast}(T)}[\operatorname*{epi}{}_{K}(F+I_{C}+T\circ
G)^{\ast}+(0\mathbf{{_{\mathcal{L}}}},v)]\\
&  \Longleftrightarrow\exists T\in\mathcal{L}_{+}^{w}(S,K):(L,y-v)\in
\operatorname*{epi}{}_{K}(F+I_{C}+T\circ G)^{\ast},\;\forall v\in I_{-S}%
^{\ast}(T)\\
&  \Longleftrightarrow\exists T\in\mathcal{L}_{+}^{w}(S,K):y-v-L(x)+F(x)+I_{C}%
(x)+(T\circ G)(x)\notin-\operatorname*{int}K,\\
&  \hspace{6.5cm}\forall x\in X,\;\forall v\in I_{-S}^{\ast}(T),\quad\\
&  \Longleftrightarrow\exists T\in\mathcal{L}_{+}^{w}(S,K):y-L(x)+F(x)+(T\circ
G)(x)\notin I_{-S}^{\ast}(T)-\operatorname*{int}K,\;\forall x\in C,\\
&  \Longleftrightarrow\exists T\in\mathcal{L}_{+}^{w}(S,K):y-L(x)+F(x)+(T\circ
G)(x)\notin\operatorname*{WSup}T(-S)-\operatorname*{int}K,\forall x\in C\\
&  \Longleftrightarrow\exists T\in\mathcal{L}_{+}^{w}(S,K):y-L(x)+F(x)+(T\circ
G)(x)\notin T(-S)-\operatorname*{int}K,\;\forall x\in C.
\end{align*}
Then, the implication $(c)\Longrightarrow(d)$ follows.\newline%
$[(d)\Longrightarrow(c)]$ Thanks to Lemma \ref{lem2.1} we only need to prove
the inclusion $"\subset"$ in $(c)$. In fact, if $(L,y)\in\operatorname*{epi}%
{}_{K}(F+I_{A})^{\ast}$, then $y-L(x)+F(x)\notin-\operatorname*{int}K,\;$for
all $x\in A,$ and by the equivalence $(b_{1})\Longleftrightarrow(b_{3})$,
there exists $T\in\mathcal{L}_{+}^{w}(S,K)$ such that
\begin{equation}
F(x)+\left(  T\circ G\right)  (x)-L(x)+y\notin T(-S)-\operatorname*{int}%
K,\;\forall x\in C.\label{marco21}%
\end{equation}
Since $T(-S)-\operatorname*{int}K=\operatorname*{WSup}%
T(-S)-\operatorname*{int}K=I_{-S}^{\ast}(T)-\operatorname*{int}K$ (see
\eqref{1.3}), it turns out that (\ref{marco21}) is equivalent to
\[
y-v-L(x)+F(x)+I_{C}(x)+(T\circ G)(x)\notin-\operatorname*{int}K,\text{
}\forall x\in X,\;\forall v\in I_{-S}^{\ast}(T),
\]
which yields $(L,y)\in\operatorname*{epi}\nolimits_{K}(F+I_{C}+T\circ
G)^{\ast}+(0_{\mathcal{L}},v)$ for all $v\in I_{-S}^{\ast}(T)$, and the aimed
inclusion follows. \medskip
\end{proof}

\begin{theorem}
[3rd characterization of Farkas lemma]\label{theoFV3} Assume all the
assumptions of Theorem \ref{theoFV1}. Then the following statements are
equivalent:\newline$(c^{\prime})$ The set
\[
\bigcup\limits_{T\in\mathcal{L}_{+}^{w}(S,K)}\left[  \bigcap\limits_{v\in
I_{-S}^{\ast}(T)}\operatorname*{epi}\nolimits_{K}(F+I_{C}+T\circ G)^{\ast
}+(0_{\mathcal{L}},v)\right]
\]
is closed in $\mathcal{L}(X,Y)\times Y;$\newline$(d)$ $\forall(L,y)\in
\mathcal{L}(X,Y)\times Y,\;(b_{1})\Longleftrightarrow(b_{3})$.
\end{theorem}

\begin{proof}
It follows from Theorem \ref{theovectorw} that
\[
\operatorname*{epi}{}_{K}(F+I_{A})^{\ast}=\operatorname*{cl}\left\{
\bigcup_{T\in\mathcal{L}_{+}^{w}(S,K)}\left[  \bigcap_{v\in I_{-S}^{\ast}%
(T)}[\operatorname*{epi}{}_{K}(F+I_{C}+T\circ G)^{\ast}%
+(0\mathbf{{_{\mathcal{L}}}},v)]\right]  \right\}  .
\]
Hence, $(c^{\prime})$ is equivalent to
\[
\operatorname*{epi}{}_{K}(F+I_{A})^{\ast}=\bigcup_{T\in\mathcal{L}_{+}%
^{w}(S,K)}\left[  \bigcap_{v\in I_{-S}^{\ast}(T)}\operatorname*{epi}{}%
_{K}(F+I_{C}+T\circ G)^{\ast}+(0\mathbf{{_{\mathcal{L}}}},v)\right]  ,
\]
and the conclusion follows from Theorem \ref{lemFVw}. \medskip
\end{proof}

\begin{theorem}
[2nd Farkas lemma]\label{theoFV4} Assume all the assumptions of Theorem
\ref{theoFV2}. Then, for all $(L,y)\in\mathcal{L}(X,Y)\times Y$, one has
$(b_{1})\Longleftrightarrow(b_{3})\mathrm{.}$
\end{theorem}

\begin{proof}
Under the assumptions of this theorem, it follows from Theorem
\ref{theovectorw2} that
\[
\operatorname*{epi}{}_{K}(F+I_{A})^{\ast}=\bigcup_{T\in\mathcal{L}_{+}%
^{w}(S,K)}\left[  \bigcap_{v\in I_{-S}^{\ast}(T)}\operatorname*{epi}{}%
_{K}(F+I_{C}+T\circ G)^{\ast}+(0\mathbf{{_{\mathcal{L}}}},v)\right]  .
\]
The conclusion now comes from Theorem \ref{lemFVw}. \medskip
\end{proof}

It should be mentioned that the Farkas-type results of the forms
[$(b_{1})\Longleftrightarrow(b_{2})$] or [$(b_{1})\Longleftrightarrow(b_{3})$]
in Theorems \ref{theoFV1}-\ref{theoFV4}, when specified to the case where
$Y=\mathbb{R},$ following the way as in \cite{DGLM16}, either cover or extend
many Farkas-type results and their stable forms in the literature, such as
\cite{BW04/05}, \cite{DGL06}, \cite{DJ}, \cite{FLN09}, etc.

\section{Optimality conditions for vector optimization problems}

Let $X,$ $Y,$ $Z,$ $F,$ $G,$ $C,$ and $A$ be as in Section 4. In this section
we provide optimality conditions for the feasible and non-trivial vector
optimization problem
\[%
\begin{array}
[c]{rrl}%
\mathrm{(VOP)}\ \  & \operatorname*{WMin}\left\{  F(x):x\in C,\;G(x)\in
-S\right\}  . &
\end{array}
\]
Using the Farkas-type results established in the last section, we get the
following optimality conditions for $\mathrm{(VOP)}$.

\begin{theorem}
[1st characterization of optimality conditions]\label{OCA1}Let $\bar{x}\in
A\cap\operatorname{dom}F$. Assume all the assumptions of Theorem
\ref{theoFV1}. Then the following statements are equivalent:\newline$(e)$ The
set
\[
\bigcup\limits_{T\in\mathcal{L}_{+}(S,K)}\operatorname*{epi}\nolimits_{K}%
(F+I_{C}+T\circ G)^{\ast}%
\]
is closed regarding $(0_{\mathcal{L}},-F(\bar{x}));$\newline$(f)$ $\bar{x}$ is
a weak solution of $\mathrm{(VOP)}$ if and only if there exists $T\in
\mathcal{L}_{+}(S,K)$ such that
\[
-F(\bar{x})\in(F+I_{C}+T\circ G)^{\ast}(0_{\mathcal{L}})+K;
\]
\newline$(g)$ $\bar{x}$ is a weak solution of $\mathrm{(VOP)}$ if and only if
there exists $T\in\mathcal{L}_{+}(S,K)$ such that
\[
F(x)+(T\circ G)(x)-F(\bar{x})\notin-\operatorname*{int}K,\;\forall x\in C.
\]

\end{theorem}

\begin{proof}
Under the current assumptions, we get from Theorem \ref{theo2} that
\[
\operatorname*{epi}{}_{K}(F+I_{A})^{\ast}=\operatorname*{cl}\left[
\bigcup_{T\in\mathcal{L}_{+}(S,K)}\operatorname*{epi}{}_{K}(F+I_{C}+T\circ
G)^{\ast}\right]  .
\]
Hence, $(e)$ is equivalent to%
\[%
\begin{array}
[c]{l}%
\operatorname*{epi}{}_{K}(F+I_{A})^{\ast}\cap\{(0_{\mathcal{L}},-F(\bar
{x}))\}\\
=\left(  \bigcup\limits_{T\in\mathcal{L}_{+}(S,K)}\operatorname*{epi}{}%
_{K}(F+I_{C}+T\circ G)^{\ast}\right)  \cap\{(0_{\mathcal{L}},-F(\bar{x}))\}
\end{array}
\]
and the conclusion follows from Theorem 5.3 in \cite{DGLM16}. \medskip
\end{proof}

\begin{theorem}
[1st optimality conditions]\label{OCB1} Let $\bar{x}\in A\cap
\operatorname{dom}F$. Assume all the assumptions of Theorem \ref{theoFV2}.
Then $(f)$ and $(g)$ hold.
\end{theorem}

\begin{proof}
Under the assumptions of this theorem, we get from Theorem \ref{theovector2}
and Lemma \ref{lem4} that $\bigcup\limits_{T\in\mathcal{L}_{+}(S,K)}%
\operatorname*{epi}{}_{K}(F+I_{C}+T\circ G)^{\ast}$ is a closed subset of
$\mathcal{L}(X,Y)\times Y$. So, $(e)$ holds and the conclusion comes from
Theorem \ref{OCA1}. \medskip
\end{proof}

\begin{theorem}
[2nd characterization of optimality conditions]\label{OCA2} Let $\bar{x}\in
A\cap\operatorname{dom}F$. Assume all the assumptions of Theorem
\ref{theoFV1}. Then the following statements are equivalent:\newline$(h)$ The
set
\[
\bigcup\limits_{T\in\mathcal{L}_{+}^{w}(S,K)}\left[  \bigcap\limits_{v\in
I_{-S}^{\ast}(T)}[\operatorname*{epi}_{K}(F+I_{C}+T\circ G)^{\ast
}+(0_{\mathcal{L}},v)]\right]
\]
is closed regarding $(0_{\mathcal{L}},-F(\bar{x}));$\newline$(i)$ $\bar{x}$ is
a weak solution of $\mathrm{(VOP)}$ if and only if there exists $T\in
\mathcal{L}_{+}^{w}(S,K)$ such that
\[
-F(\bar{x})-I_{-S}^{\ast}(T)\subset(F+I_{C}+T\circ G)^{\ast}(0_{\mathcal{L}%
})+K;
\]
\newline$(j)$ $\bar{x}$ is a weak solution of $\mathrm{(VOP)}$ if and only if
there exists $T\in\mathcal{L}_{+}^{w}(S,K)$ such that
\[
F(x)+(T\circ G)(x)-F(\bar{x})\notin T(-S)-\operatorname*{int}K,\;\forall x\in
C.
\]

\end{theorem}

\begin{proof}
Arguing as in the proof of Theorem \ref{lemFVw}, with $(0_{\mathcal{L}%
},-F(\bar{x}))$ instead of $(L,y),$ we get%
\begin{equation}%
\begin{array}
[c]{l}%
(0_{\mathcal{L}},-F(\bar{x}))\in\bigcup_{T\in\mathcal{L}_{+}^{w}(S,K)}\left[
\bigcap_{v\in I_{-S}^{\ast}(T)}[\operatorname*{epi}{}_{K}(F+I_{C}+T\circ
G)^{\ast}+(0\mathbf{{_{\mathcal{L}}}},v)]\right] \\
\Longleftrightarrow\;\exists T\in\mathcal{L}_{+}^{w}(S,K):(0_{\mathcal{L}%
},-F(\bar{x})-v)\in\operatorname*{epi}{}_{K}(F+I_{C}+T\circ G)^{\ast
},\;\forall v\in I_{-S}^{\ast}(T)\\
\Longleftrightarrow\;\exists T\in\mathcal{L}_{+}^{w}(S,K):-F(\bar{x}%
)-v\in(F+I_{C}+T\circ G)^{\ast}(0_{\mathcal{L}})+K,\;\forall v\in I_{-S}%
^{\ast}(T)\\
\Longleftrightarrow\;\exists T\in\mathcal{L}_{+}^{w}(S,K):-F(\bar{x}%
)-I_{-S}^{\ast}(T)\subset(F+I_{C}+T\circ G)^{\ast}(0_{\mathcal{L}})+K,
\end{array}
\label{35}%
\end{equation}
and also
\begin{equation}%
\begin{array}
[c]{l}%
(0_{\mathcal{L}},-F(\bar{x}))\in\bigcup_{T\in\mathcal{L}_{+}^{w}(S,K)}\left[
\bigcap_{v\in I_{-S}^{\ast}(T)}[\operatorname*{epi}{}_{K}(F+I_{C}+T\circ
G)^{\ast}+(0\mathbf{{_{\mathcal{L}}}},v)]\right] \\
\Longleftrightarrow\;\exists T\in\mathcal{L}_{+}^{w}(S,K):F(x)+(T\circ
G)(x)-F(\bar{x})\notin T(-S)-\operatorname*{int}K,\;\forall x\in C.
\end{array}
\label{36}%
\end{equation}
\newline Next, under the assumptions of this theorem, we get from Theorem
\ref{theovectorw} that
\[
\operatorname*{epi}{}_{K}(F+I_{A})^{\ast}=\operatorname*{cl}\left\{
\bigcup_{T\in\mathcal{L}_{+}^{w}(S,K)}\left[  \bigcap_{v\in I_{-S}^{\ast}%
(T)}[\operatorname*{epi}{}_{K}(F+I_{C}+T\circ G)^{\ast}%
+(0\mathbf{{_{\mathcal{L}}}},v)]\right]  \right\}  .
\]
Hence, $(h)$ is equivalent to
\begin{align}
& \operatorname*{epi}{}_{K}(F+I_{A})^{\ast}\cap\{(0_{\mathcal{L}},-F(\bar
{x}))\}=\label{37}\\
& \hskip0.5cm\left\{  \bigcup_{T\in\mathcal{L}_{+}^{w}(S,K)}\left[
\bigcap_{v\in I_{-S}^{\ast}(T)}[\operatorname*{epi}{}_{K}(F+I_{C}+T\circ
G)^{\ast}+(0\mathbf{{_{\mathcal{L}}}},v)]\right]  \right\}  \cap
\{(0_{\mathcal{L}},-F(\bar{x}))\}.\nonumber
\end{align}
Now, take an arbitrary $\bar{x}\in A\cap\operatorname{dom}F$. Let us recall
that $\bar{x}$ is a weak solution of (VOP) if and only if $(0_{\mathcal{L}%
},-F(\bar{x}))\in\operatorname*{epi}{}_{K}(F+I_{A})^{\ast}$ (see
\cite[Proposition 5.1]{DGLM16}). Hence, we get [$(h)\Longleftrightarrow(i)$]
from \eqref{37} and \eqref{35}, and [$(h)\Longleftrightarrow(j)$] from
\eqref{37} and \eqref{36}. \medskip
\end{proof}

\begin{theorem}
[2nd optimality conditions]\label{OCB2} Let $\bar{x}\in A\cap
\operatorname{dom}F$. Assume all the assumptions of Theorem \ref{theoFV2}.
Then $(i)$ and $(j)$ hold.
\end{theorem}

\begin{proof}
Under the current assumptions, we get from Theorem \ref{theovectorw2} and
Lemma \ref{lem4} that
\[
\bigcup_{T\in\mathcal{L}_{+}^{w}(S,K)}\left[  \bigcap_{v\in I_{-S}^{\ast}%
(T)}[\operatorname*{epi}{}_{K}(F+I_{C}+T\circ G)^{\ast}%
+(0\mathbf{{_{\mathcal{L}}}},v)]\right]
\]
is a closed subset of $\mathcal{L}(X,Y)\times Y$. So, $(h)$ holds and the
conclusion comes from Theorem \ref{OCA2}. \medskip
\end{proof}

We now revisit Example \ref{Example2} paying attention to statements $(f)$ and
$(i)$, whose common right-hand side set is
\[
(F+I_{C}+(t_{1},t_{2})\circ G)^{\ast}\left(  0_{\mathcal{L}}\right)
+\mathbb{R}_{+}^{2}=\left\{
\begin{array}
[c]{ll}%
\left(  \mathbb{R\times R}_{+}\right)  \cup\left(  \mathbb{R}_{+}%
\mathbb{\times R}\right)  , & \text{if }t_{1}=t_{2}=0,\\
\left\{  +\infty_{\mathbb{R}^{2}}\right\}  , & \text{if }t_{1}t_{2}>0,\\
\mathbb{R}\left(  t_{1},t_{2}\right)  +\mathbb{R}_{+}^{2}, & \text{else.}%
\end{array}
\right.
\]
It can be easily realized that the elements of $\mathcal{L}_{+}(S,K)$ and
$\mathcal{L}_{+}^{w}(S,K)$ satisfying the optimality conditions in $(f)$ and
$(i)$ are those of%
\[
\left\{  \left(  t_{1},t_{2}\right)  \in\mathbb{R}_{+}^{2}:t_{1}t_{2}%
\leq0\right\}  =\left[  \mathbb{R}_{+}\times\{0\}\right]  \cup\left[
\{0\}\times\mathbb{R}_{+}\right]
\]
and%
\[
\left\{  \left(  t_{1},t_{2}\right)  \in\left(  \mathbb{R\times R}_{+}\right)
\cup\left(  \mathbb{R}_{+}\times\mathbb{R}\right)  :\,t_{1}t_{2}\leq0\right\}
=\left[  -\left(  \mathbb{R}_{+}\right)  \times\mathbb{R}_{+}\right]
\cup\left[  \mathbb{R}_{+}\times-\left(  \mathbb{R}_{+}\right)  \right]  ,
\]
respectively. Since both sets are nonempty, we get the trivial conclusion (as
$F\equiv0_{Y}$) that any feasible solution $\overline{x}$ is weakly
minimal.\medskip

It is worth mentioning that the Farkas-type results in Section 4 and the
optimality conditions in this section can be used to derive duality results
for \textrm{(VOP)}. For instance, the corollary proved below extends the
strong duality result \cite[Theorem 4.2.7]{BGW09}, which was established under
the assumption that $(c_{1})$ holds.

In \cite[p. 138]{BGW09}, the \emph{dual problem} of $\mathrm{(VOP)}$
\emph{with respect to weakly efficient solutions} is defined as:
\[
(\mathrm{DVOP})\quad\operatorname*{WMax}\left\{  y:(T,y)\in B\right\}  ,
\]
where the \emph{dual feasible set} $B$ is the set of pairs $(T,y)\in
\mathcal{L}_{+}(S,K)\times Y$ such that there is no $x\in C\cap
\operatorname*{dom}G$ such that $(F+T\circ G)(x)<_{K}y.\ $Equivalently,
\[
B=\{(T,y)\in\mathcal{L}_{+}(S,K)\times Y:y\notin(F+T\circ G)(C\cap
\operatorname*{dom}G)+\operatorname*{int}K\}.
\]

\begin{corollary}
[Strong duality for the pair \textrm{(VOP)} - \textrm{(DVOP)}]Let
$F\colon\ X\rightarrow Y\cup\{+\infty_{Y}\}$ be a proper $K$-convex mapping,
$C$ be a nonempty convex subset of $X$, and $G\colon X\rightarrow
Z\cup\{+\infty_{Z}\}$ be a proper $S$-convex mapping. Assume that
$A\cap\operatorname*{dom}F\neq\emptyset$ and one of the conditions $(c_{1})$,
$(c_{2})$, or $(c_{3})$ is fulfilled. If $\bar{x}\in A$ is a solution of
$\mathrm{(VOP),}$ then there exists a solution $(\overline{T},\bar{y})$ of
$(\mathrm{DVOP})$ such that $F(\bar{x})=\bar{y}$.
\end{corollary}

\begin{proof}
Denote $M:=\{y:(T,y)\in B\}$ and assume that $\bar{x}$ is a solution of
$\mathrm{(VOP)}$. It follows from Theorem \ref{OCB1} that $(g)$\textrm{\ }%
holds, i.e., there exists $\overline{T}\in\mathcal{L}_{+}(S,K)$ such that
$(\overline{T},F(\bar{x}))\in B$. Hence, $F(\bar{x})\in M$.\newline Now, take
an arbitrary $y\in M$. By the definition of the set $M$,
\begin{equation}
\exists T_{0}\in\mathcal{L}_{+}(S,K):y\notin(F+T_{0}\circ G)(C\cap
\operatorname*{dom}G)+\operatorname*{int}K.\label{5.37}%
\end{equation}
Since $\bar{x}\in A$, $G(\bar{x})\in-S$ and so $-T_{0}\circ G(\bar{x})\in K$.
On the other hand, one gets from \eqref{5.37} that $[y-F(\bar{x}%
)]+[-(T_{0}\circ G)(\bar{x})]\notin\operatorname*{int}K$, and hence,
$y-F(\bar{x})\notin\operatorname*{int}K$ as $\operatorname*{int}%
K+K=\operatorname*{int}K$. Since this holds for any $y\in M$, one gets
$F(\bar{x})\notin M-\operatorname*{int}K$.\newline We have just shown that
$F(\bar{x})\in M\setminus(M-\operatorname*{int}K)$, which yields that
$M\neq\emptyset$ and $\operatorname*{WSup}M\neq\{+\infty_{Y}\}$ (see
\eqref{wsup=infty}). It now follows from \eqref{eqmax} that $F(\bar{x}%
)\in\operatorname*{WMax}M$. So $(\overline{T},F(\bar{x}))$ is a solution of
$(\mathrm{DVOP})$ and the proof is complete. \medskip
\end{proof}

\section*{Acknowledgments}

This research was supported by the National Foundation for Science \&
Technology Development (NAFOSTED), from Vietnam, Project 101.01-2015.27
\emph{Generalizations of Farkas lemma with applications to optimization,} by
the Ministry of Economy and Competitiveness of Spain, and by the FEDER of EU,
Project MTM2014-59179-C2-1-P, and by the Australian Research Council, Project DP160100854.

\end{document}